\documentclass{amsart}

\usepackage{amsmath,amssymb,amsthm, mathrsfs, mathtools}
\usepackage{amscd,mathtools}
  \usepackage{geometry}
  \usepackage{graphicx,epstopdf}
  \usepackage{color,xcolor}
  \usepackage{enumerate}
  \usepackage{xspace}
  \usepackage{tipa}
  \usepackage{appendix}
  \usepackage{xcolor}

\usepackage{pgf,tikz}
\usetikzlibrary{arrows}
\usepackage[all]{xy}
\usepackage{tikz-cd}
\tikzcdset{arrow style=tikz,diagrams={>=stealth}}
\usetikzlibrary{patterns.meta}

\tikzstyle{stuff_fill}=[rectangle,fill=white,minimum size=1em]
\tikzstyle{stuff_fillc}=[circle,fill=white,minimum size=1em]

  \usepackage{epsfig}
  \usepackage{pinlabel}
  \usepackage[all]{xy}
  
\usepackage{tikz}
\usepackage{caption, subcaption}
\usepackage{tikz-cd}
\usetikzlibrary{calc,decorations.pathreplacing}
\usepackage{tikz}
\usetikzlibrary{calc}
\usetikzlibrary{arrows}
\usetikzlibrary{decorations.pathreplacing}
\usetikzlibrary{intersections}

\usepackage{mathrsfs}
\usetikzlibrary{matrix}
\usetikzlibrary{positioning}
\DeclareMathAlphabet{\pazocal}{OMS}{zplm}{m}{n}

\usepackage{mathrsfs}
\usetikzlibrary{matrix}
\usetikzlibrary{positioning}
\DeclareMathAlphabet{\pazocal}{OMS}{zplm}{m}{n}
\tikzset{>=stealth}

\RequirePackage{hyperref}
\hypersetup{pdfpagemode={UseOutlines},
bookmarksopen=true,
bookmarksopenlevel=0,
hypertexnames=false,
colorlinks=true, 
citecolor=teal, 
linkcolor=blue, 
urlcolor=magenta, 
pdfstartview={FitV},
unicode,
breaklinks=true,
}

\newtheorem{theorem}{Theorem}[section]
\newtheorem{lemma}[theorem]{Lemma}
\newtheorem{proposition}[theorem]{Proposition}

  \newtheorem{introthm}{Theorem}
  
  \renewcommand{\theintrothm}{\Alph{introthm}}

\newtheorem*{theorem*}{Theorem}
\newtheorem*{ques*}{Question}
\newtheorem*{prop*}{Proposition}

\theoremstyle{definition}
\newtheorem{definition}[theorem]{Definition}

\newtheorem{claim}[theorem]{Claim}

\newtheorem*{claim*}{Claim}

\newtheorem*{question*}{Question}
\newtheorem*{answer*}{Answer}
\newtheorem*{application*}{Application}

\theoremstyle{remark}
\newtheorem{remark}[theorem]{Remark}

\numberwithin{equation}{section}

  \newcommand{\calN}{\mathcal{N}}

  \newcommand{\calT}{\mathcal{T}}
  \newcommand{\calU}{\mathcal{U}}
    \newcommand{\calV}{\mathcal{V}}

  \newcommand{\NN}{\mathbb{N}}

  \newcommand{\RR}{\mathbb{R}}

  \newcommand{\bfa}{\textbf{a}}
  \newcommand{\bfb}{\textbf{b}}

  \newcommand{\gothic}{\mathfrak}

  \newcommand{\go}{{\gothic o}}

  \newcommand{\sD}{{\sf D}}

  \newcommand{\sQ}{{\sf Q}}   
  \newcommand{\sR}{{\sf R}}

  \newcommand{\mm}{{\sf m}}   
  \newcommand{\nn}{{\sf n}}

  \newcommand{\qq}{{\sf q}}   
  \newcommand{\rr}{{\sf r}}

\DeclareMathOperator{\diam}{diam}

  \newcommand{\eqnref}[1]{Equation~\eqref{#1}}

  \newcommand{\param}{{\mathchoice{\mkern1mu\mbox{\raise2.2pt\hbox{$
  \centerdot$}}
  \mkern1mu}{\mkern1mu\mbox{\raise2.2pt\hbox{$\centerdot$}}\mkern1mu}{
  \mkern1.5mu\centerdot\mkern1.5mu}{\mkern1.5mu\centerdot\mkern1.5mu}}}

\DeclarePairedDelimiterX{\norm}[1]{\lvert}{\rvert}{#1}
\DeclarePairedDelimiterX{\Norm}[1]{\lVert}{\rVert}{#1}

  \newcommand{\from}{\colon\thinspace}

  \newcommand{\pka}{{ \partial_{\kappa} }}

\newcommand{\CAT}{\ensuremath{\operatorname{CAT}(0)}\xspace}        
\begin{document}
\bibliographystyle{alpha}

\title[SBE and sublinearly Morse boundaries]{Sublinear biLipschitz equivalence and sublinearly Morse boundaries}




\author   {Gabriel Pallier}
\address{Karlsruhe Institute of Technology, 76131 Karlsruhe, Germany}
\email{gabriel@pallier.org}

\thanks{G.P. gratefully acknowledges funding by the DFG 281869850 (RTG 2229) and the ERC Starting Grant 713998 GeoMeG ‘Geometry of Metric Groups’.}

\author   {Yulan Qing}
\address{Shanghai Center for Mathematical Sciences, Fudan University, Shanghai, China}
\email{yulan.qing@gmail.com}


\begin{abstract}
A sublinear biLipschitz equivalence (SBE) between metric spaces is a map from one space to another that distorts distances with bounded multiplicative constants and sublinear additive error. Given any sublinear function $\kappa$, $\kappa$-Morse boundaries are defined for all geodesic proper metric spaces as a quasi-isometrically invariant and metrizable topological space of quasi-geodesic rays. In this paper, we prove that $\kappa$-Morse boundaries of proper geodesic metric spaces are invariant under suitable SBEs. 
A tool in the proof is the use of sublinear rays, that is, sublinear bilispchitz embeddings of the half line, generalizing quasi-geodesic rays.
As an application we distinguish a pair of right-angled Coxeter groups brought up by Behrstock up to sublinear biLipschitz equivalence.
We also show that under mild assumptions, generic random walks on countable groups are sublinear rays.
\end{abstract}
\maketitle

%

\section{Introduction}

Sublinear biLipschitz equivalence is an equivalence relation between metric spaces that naturally generalizes quasiisometries.  It appeared first without a name, and then more explicitly in the work of Cornulier on asymptotic cones of Lie groups \cite{Cor08,Cor11}, before being studied for its own sake \cite{cornulier2017sublinear,pallierHYPSBE}. 
Following Cornulier we will abbreviate this relation as SBE. Because Gromov hyperbolicity admits a characterization in terms of asymptotic cones, it follows that being hyperbolic  is an SBE-invariant property among compactly generated locally compact groups; this was noted by Cornulier \cite[Theorem 4.3]{cornulier2017sublinear}.

In this project we extend this result beyond groups and spaces that are Gromov hyperbolic.
Precisely, given a proper geodesic space that is not Gromov hyperbolic, but exhibiting some features of Gromov hyperbolic spaces, one can study its large scale hyperbolic-like structure by describing the \emph{sublinearly Morse boundaries} of the group \cite{QRT20}.
The latter is a topological space collecting a large set of quasi-geodesic rays behaving in a hyperbolic fashion; when the space is Gromov-hyperbolic, all of its sublinearly Morse boundaries are homeomorphic to its Gromov boundary. 
In this paper we show that the set of subblinearly Morse directions are invariant under suitable sublinear biLipschitz equivalence.
We then go further to obtain simultaneous generalizations of two distinct previous results in the literature : 
\begin{enumerate}
\item 
(See \S \ref{subsec:boundaries} below)
Qing, Rafi and Tiozzo's theorem that the homeomorphism type of the sublinearly Morse boundary is a quasi-isometry invariant among proper geodesic metric spaces \cite[Theorem A(2)]{QRT20}.
\item 
(See \S \ref{subsec:sbe-intro} below)
Cornulier's theorem that the {homeomorphism type of the} Gromov boundary is SBE invariant among Gromov-hyperbolic groups \cite{cornulier2017sublinear}.
\end{enumerate}
We recall some context on these two theorems.

\subsection{Boundaries}
\label{subsec:boundaries}

In his seminal article \cite{Gromov87HypGrp}, Gromov introduced the class of hyperbolic groups and attached to such groups an equivariant bordification, now called the Gromov boundary. The class of Gromov-hyperbolic group is closed under quasiisometry, and the quasiisometries extend equivariantly to the Gromov boundaries.

The class of Gromov hyperbolic groups and spaces is however not vast enough to include natural examples such as $\mathrm{CAT}(0)$ groups and mapping class groups of surfaces of finite type.

For $\mathrm{CAT}(0)$ groups and spaces, the visual boundary (the set of all geodesic rays emanating from a fixed base-point $(X, \go)$, up to fellow travel) does not provide a good large-scale invariant (as indicated by the works from Croke-Kleiner \cite{CK00} to Qing \cite{qing1}). In \cite{QRT19}, the second named author and Rafi consider the set of quasi-geodesic rays whose Morse property is weakened compared to that of geodesic rays in Gromov spaces. In particular, given a sublinear function $\kappa$, Qing-Rafi define a quasi-geodesic ray $\gamma$ to be \emph{sublinearly $\kappa$-Morse} if any other geodesic segment with endpoints on $\gamma$ is uniformly $\kappa$-close to $\gamma$, i.e. their distances to $\gamma$ is bounded above by $n(\qq, \sQ) \kappa(\Norm x)$, where $n(\qq, \sQ)$ is a constant depends only on the quasi-geodesic constants $(\qq, \sQ)$ of the segment, and the distance of each point on the segment to the origin, $\Norm x = d(\go, x)$. The collection of all such quasi-geodesic rays, together with a coarse cone topology, is referred to as the $\kappa$-boundary of $X$, and denoted $\pka X$. These boundaries are shown to be quasi-isometrically invariant  topological spaces attached to all proper geodesic spaces \cite{QRT20}. Therefore one can denote a $\kappa$-boundary of a group with $\pka G$. Furthermore, they are metrizable topological spaces (\cite{QRT20}). Since their introduction, sublinearly Morse boundaries are studied and compared to Gromov boundaries in various ways, such as via visibility, divergence and contracting properties (See\cite{MQZ21}, \cite{IZ} and \cite{Zal21} ).

One important application of the sublinear boundaries is that, for appropriately chosen $\kappa$, $\pka G$ is a topological model for the Poisson boundaries of simple random walks on various groups, such as right-angled Artin groups \cite{QRT19}, mapping class groups and relative hyperbolic groups \cite{QRT20}, hierarchically hyperbolic groups \cite{NQ22},  CAT(0) groups \cite{GQR22} and Teichm\"uller spaces \cite{GQR22}. The sublinearly Morse directions are also shown to be generic in Patterson Sullivan measure under suitable conditions \cite{GQR22}. Most recently, Choi \cite{Choi} claim this result to hold for all groups with two independent isometries with contracting axes.

\subsection{Sublinear biLipschitz equivalence}
\label{subsec:sbe-intro}
Sublinear biLipschitz equivalence (SBE) appeared in works of Cornulier, where it was motivated by the quasiisometry classification of connected Lie groups. 
Cornulier noted that while the quasiisometry classification of all such groups reduces to that of closed subgroups of real upper triangular matrices, the sublinear biLipschitz equivalence classification reduces to that of a smaller class, and that it was completely treated by the literature in the nilpotent case \cite{Cor11}. In \cite{cornulier2017sublinear} Cornulier asked about the SBE classification for other classes of groups, especially the word-hyperbolic groups.

As mentioned above, the first SBE invariant is the asymptotic cone; it is also the most natural one, since SBE may be defined as the largest class of maps inducing biLipschitz homeomorphisms between asymptotic cones with fixed base-points \cite{Cor11}.
A geodesic metric space $X$ is Gromov-hyperbolic if and only if all its asymptotic cones are real trees, for every choice of sequence of base-points \cite[2.A]{AsInv} \cite[Proposition 3.A.4]{Drutu}.  It follows that Gromov-hyperbolicity is an SBE-invariant among compactly generated locally compact groups, and especially among finitely generated groups)\cite[Theorem 4.3]{cornulier2017sublinear}.
Cornulier proved that SBEs between Gromov-hyperbolic groups induce biH\"older homeomorphisms between their Gromov boundaries; this was slightly improved by the first named author showing that a sublinear conformal structure in the Gromov boundary is actually preserved \cite{cornulier2017sublinear,pallierHYPSBE}. In this paper we extend the topological invairance of the Gromov boundary to that of a family of topological spaces that can be attached to any proper geodesic metric spaces:

\begin{introthm}(Theorem~\ref{invarianttopology})\label{th:invariant-topology-intro}
%
Let $L \geqslant 1$, and let $\theta$ be a concave, nondecreasing and stictly sublinear function. Consider proper geodesic metric spaces $X$ and $Y$, let $\Phi \from X \to Y$ be a 
$(L, \theta)$-sublinear bi-Lipschitz equivalence between $X$ and $Y$.  Then, for every concave, nondecreasing and strictly sublinear function $\kappa$ that dominates $\theta$,  there is a homeomorphism $\Phi_\star \from \partial_\kappa X \to \partial_\kappa Y$,  induced by $\Phi$. 

\end{introthm}

In the statement when we say that $\kappa$ {dominates} $\theta$ we mean that there exists constants $C_1, C_2$ and a nonnegative real number $t_0$ such that for all $t>t_0$, $\kappa(t) \geq C_1 \theta(t) + C_2$. For the complete definitions of the $\kappa$-Morse boundary $\partial_\kappa$ and of a $(L,\theta)$-sublinear biLipschitz equivalence we refer to Subsection \ref{subsec:dfn-pka-boundaries} and Definition \ref{defn:theta-SBE} respectively.
When $\theta= 1$, $\Phi$ is a quasiisometry and Theorem~\ref{th:invariant-topology-intro} is exactly \cite[Theorem A(2)]{QRT20}; when $X$ and $Y$ are Gromov hyperbolic, Theorem~\ref{th:invariant-topology-intro} is a consequence of Theorem 1.7 in \cite{cornulier2017sublinear}.

\subsection*{Application} As an application, we use Theorem \ref{th:invariant-topology-intro} to distinguish two non-relatively hyperbolic right-angled Coxeter groups up to SBE in Section \ref{sectionboundary}. This pair was presented by Behrstock in \cite{Be19}. 
Determining whether these two groups have bilipschitz homeomorphic asymptotic cones or not is an open question \cite{BeP}.
On the other hand, we check in \ref{subsec:application-of-thA} that their sublinearly Morse boundaries have different topological dimensions, which allows us to conclude from Theorem \ref{th:invariant-topology-intro} that they cannot be sublinear biLipschitz equivalent.

The proof of Theorem~\ref{invarianttopology} makes use of sublinear rays, introduced in \cite{pallierHYPSBE} (there called $(\lambda, O(v))$-rays). These are images of the half-line under a sublinear bi-Lipschitz embedding. As such, they are analogues of quasi-geodesics, but with an additive constant of quasigedoesicity that grows sublinearly. If the growth if logarithmic, we call the resulting sublinear ray a logarithmic ray.
Sublinear rays also play an important role in the statement of our second result.

\subsection{Random walks}
Similar to the development of sublinearly Morse boundaries, an important motivation behind this project stems from simple random walk on finitely generated groups. In \cite{Ti15}, Tiozzo show that given a surface $S$, in the Teichm\"uller space $\calT(S)$, a generic random walk tracks a geodesic sublinearly. 
In Section~\ref{sec:random-walks} of this paper, we show that this implies that a generic random walk is a sublinear ray.

\begin{introthm}[Theorem~\ref{randomwalkisthetaray}]\label{th:randomwalkintro}
Let $G$ be the mapping class group $\operatorname{Mod}(S)$ of a finite type surface, or let $G$ be a finitely generated relatively hyperbolic group. 
Let $\mu$ be a probability measure on $G$ with finite first moment with respect to any word metric on $G$, such that the subsemigroup generated by the support of $\mu$ is a non-amenable group. Then almost every sample path of the random walk associated to $\mu$ is a logarithmic ray.
\end{introthm}

In fact, a version of this result with a weaker conclusion holds for a larger class of group actions: we combine the proof of Theorem~\ref{randomwalkisthetaray} together Theorem 6 in \cite{Ti15} to obtain the following.

\setcounter{introthm}{1}
\renewcommand{\theintrothm}{\Alph{introthm}'}

\begin{introthm}[Theorem~\ref{countablegroups}]\label{th:randomwalkintro2}
 Let $G$ be a countable group acting via isometries on a proper, geodesic, metric space $(X,d)$ with a non-trivial, stably visible compactification. Let $\go$ denote a basepoint in $X$. Let $\mu$ be a probability measure on G with finite first moment with respect to $d$, such that the subsemigroup generated by the support of $\mu$ is a non-amenable group. Then for almost every sample path $(w_n)$ of the random walk associated to $\mu$, $(w_n \cdot \go)$ is a sublinear ray.
\end{introthm}

While Theorem~\ref{randomwalkisthetaray} is logically independent from Theorem~\ref{invarianttopology}, both theorems promote the use of sublinear rays. Whereas a generic random walk is not a quasi-geodesic, it is a sublinear ray, and many of the geometric techniques devised for quasi-geodesics can be employed to treat sublinear rays as if they were quasi-geodesics.

\subsection{Organization of the paper}
Section \ref{sec:prelims} collects preliminary information; especially we recall the relevant definitions and facts concerning sublinear biLipschitz equivalence, rays, and the sublinearly Morse boundaries. 
Beware that we chose to adopt the notation from \cite{QRT20} so that the notation for sublinear biLipschitz equivalence is not the usual one.
Section \ref{sec:SBE-Morse-ray-to-Morse-ray} is a preparation for Section \ref{sectionboundary}, which is itself devoted to the proof of Theorem~\ref{invarianttopology}.
Section \ref{sec:random-walks} deals with Theorem \ref{th:randomwalkintro} and Theorem \ref{th:randomwalkintro2}. It only builds on the preliminaries and can be  read without Sections \ref{sec:SBE-Morse-ray-to-Morse-ray} and \ref{sectionboundary}.

\section{Preliminaries}
\label{sec:prelims}

\subsection{Notation and convention for the sublinear functions}
\label{subsec:notation}
Throughout this paper, let $X$ and $Y$ denote pointed proper geodesic metric spaces.
The base-points in both spaces are denoted $\mathfrak o$.
The distance to the base-point is denoted $\ \Vert x \Vert = d(x, \mathfrak o)$ for all $x \in X$ or $x \in Y$. Let $\kappa$ be a concave {nondecreasing} and strictly sublinear function. The last condition means that $\kappa(r) /r$ goes to $0$ as $r $ tends to $+\infty$. {We also assume $\kappa \geqslant 1$}.

By $\kappa(x)$ for $x \in X$ or $Y$ we mean $\kappa(\Vert x \Vert)$.
Let $Z \subseteq X$ be a closed subspace of $X$ and $\sD >0$, then
$\mathcal N_\kappa(Z, \sD)$ will denote $\{ x\in X : d(x,Z) \leqslant \sD \kappa(x) \}$. 
We say that $\sD$ is {\em small with respect to $r$}, and write $\sD \ll r$, if $\sD \leqslant r / (2 \kappa(r))$.

\subsection{Sublinear estimates}

Here is a basic sublinear estimate that we need:

\begin{lemma}[Sublinear Estimation Lemma]

\label{Lem:sublinear-estimate}
For any $\sD_0>0$, there exists  $\sD_{1}, \sD_{2} > 0$ depending on $\sD_0$
and $\kappa$ so that, for $x, y \in X$,
\[
d(x, y) \leq \sD_{0} \cdot \kappa(x)
\Longrightarrow
\sD_{1} \kappa(x) \leq \kappa(y) \leq \sD_{2} \kappa(x).
\]
\end{lemma}

\begin{proof}
Since $\kappa$ is sublinear, there is $R > 0$ such that $\kappa(x) \leqslant \frac{1}{2\sD_0} \Vert x \Vert$ as soon as $\Vert x \Vert \geqslant R$.
And then $d(x,y) \leqslant \sD_0 \kappa(x)$ implies that $\Vert y \Vert \leqslant 3 \Vert x \Vert/2$ by the triangle inequality, so that, in all cases,
\[ \kappa(y) \leqslant \left[ \sup_{r \geqslant R} \frac{\kappa(3r/2)}{\kappa(r)} + \kappa(3R/2) \right] \kappa(x) \]
where we used that $\kappa(x) \geqslant 1$. 
We may define $\sD_2 = \sup_{r \geqslant R} \frac{\kappa(3r/2)}{\kappa(r)} + \kappa(3R/2)$.
On the other hand, if $\Vert x \Vert \geqslant R$ then $\Vert y \Vert \geqslant \Vert x \Vert /2$ so that
\[ \kappa(y) \geqslant \left[ \inf_{r \geqslant R} \frac{\kappa(r/2)}{\kappa(r)} \right] \kappa(x) \]
Setting $\sD_1 = \min( \inf_{r \geqslant R} \frac{\kappa(r/2)}{\kappa(r)} , 1/\kappa(R))$ finishes the proof.
\end{proof}

\subsection{Quasi-geodesics and $\theta$-rays}

Here, $\theta$ is a function with the same properties as $\kappa$, that were specified in \S\ref{subsec:notation}.

\begin{definition}[Sublinear ray]\label{dfn:theta-ray}
Let $X$ be a proper pointed geodesic metric space. Let $L \geqslant 1$ be a constant.
Say that $\gamma : [0,+\infty) \to X$ with $\gamma(0)=\mathfrak o$ is a $(L,\theta)$-ray, or a $\theta$-ray for short, if for every $s,t \in [0, + \infty)$
\begin{equation}
    \frac{1}{L}\vert s - t \vert - \theta(\max (s,t)) \leqslant d(\gamma(s), \gamma(t)) \leqslant L \vert s - t \vert + \theta(\max (s,t)).
    \label{eq:kappa-ray}
\end{equation} 
\end{definition}

If we do not want to specify $\theta$, we will just say that $\gamma$ as in the definition above is a sublinear ray.
Beware that in \cite{pallierHYPSBE} we did not ask $\gamma(0)=\mathfrak o$ in the definition of a sublinear ray but we do it here. The difference is not a serious one, as one may simply advance the function $\theta$ (i.e. replace $\theta$ with $\theta (\mathsf D + \cdot)$) to accommodate for the change.

\begin{remark}[compare {\cite[Lemma 3.2]{pallierHYPSBE}}]
\label{rem:linear-control}
If $\gamma$ is a sublinear ray with large-scale Lipschitz constant $L$ and sublinear function $\theta$, then for $s$ large enough $\frac{1}{2L} \vert s \vert \leqslant \Vert \gamma(s) \Vert \leqslant 2L \vert s \vert$ (apply \eqref{eq:kappa-ray} with fixed $t$), hence there exists $\widehat \theta = O(\theta)$ such that for $s$ and $t$ large enough,
\begin{equation}
    \frac{1}{L} \vert s - t \vert - \widehat \theta(\max (\Vert \gamma(s)\Vert,\Vert \gamma(t)\Vert)) \leqslant d(\gamma(s), \gamma(t)) \leqslant L\vert s - t \vert + \widehat  \theta(\max (\Vert \gamma(s)\Vert,\Vert \gamma(t)\Vert)).
    \label{eq:kappa-ray-reformulated}
\end{equation}
\end{remark} 

When $\theta$ is a constant, the Definition \ref{dfn:theta-ray} is that of a quasi-geodesic ray.
For our purposes, it is however necessary to treat the latter specifically because they play a special role in the definition of the sublinearly Morse boundary.
Hence, we will use $q$ for $L$ and $Q$ for $\theta$ when we want to denote a quasi-geodesic ray.

\begin{lemma}[Connect-the-dots for $\theta$-rays]\label{lem:connect-dots}
Let $\gamma$ be a $(L, \theta)$-ray in a proper geodesic metric space $X$.
Then there exists $n >0$ and $\widehat \gamma$ which is a $(L,n \cdot \theta)$-ray in $X$ with the property that 
\begin{itemize}
\item 
$\gamma(t) = \widehat \gamma(t)$ for all nonnegative integer $t$.
\item 
$\widehat \gamma$ is continuous.
\end{itemize}
Moreover, there exists $n >0$ such that
\begin{equation}
\label{eq:completion-is-close}
d(\gamma(t), \widehat{\gamma}(t)) \leqslant n \cdot \theta(t)
\end{equation}
for all $t$. 
\end{lemma}

We will refer to $\widehat{\gamma}$ as a continuous completion of $\gamma$.

\begin{proof}
For every $t \in \mathbf N$, choose a geodesic segment $\sigma_t$ from $\gamma(t)$ to $\gamma(t+1)$ at unit speed and denote its length $\ell_t$.
Note that $\ell(t) \leqslant L+\theta(t+1)$ by the inequality on the right in \eqref{eq:kappa-ray}.
Now for all $t \in [0,+\infty)$ set $\widehat \gamma(t) = \sigma_t(\ell_t \cdot \{t \})$ where $\{ t \}$ denotes the fractional part of $t$. 
In this way $\widehat \gamma$ is continuous by construction. Let $t,s \in [0,+\infty)$ be such that $t \leqslant s$.
Then, either $\lfloor s \rfloor = \lfloor t \rfloor$, in which case $d(\widehat (\gamma(t), \gamma(s)) \leqslant \ell_{\lfloor t \rfloor} \leqslant \theta(t+1) \leqslant n_0 \theta(t)$ for some $n_0$ by the properties of $\theta$, or
\begin{align*}
d(\gamma(t), \gamma(s)) & \leqslant d(\gamma(t), \gamma(\lceil t \rceil)) + d(\gamma(\lceil t \rceil), \gamma(\lfloor s \rfloor)) + d(\gamma(\lfloor s \rfloor), \gamma(s)) \\
& \leqslant L (\lfloor s \rfloor - \lceil t \rceil)) + 2 \theta (s+1) \\
& \leqslant L \vert s - t \vert + 2 \theta (s+1).
\end{align*}
and \begin{align*}
d(\gamma(t), \gamma(s)) & \geqslant  d(\gamma(\lceil s \rceil), \gamma(\lfloor t \rfloor)) - d(\gamma(s), \gamma(\lceil s \rceil)) - d(\gamma(\lfloor t \rfloor), \gamma(t)) \\
& \geqslant L^{-1} (\lceil s \rceil - \lfloor t \rfloor)) - 2 \theta (s+1) \\
& \geqslant L^{-1} \vert s - t \vert + 2 \theta (s+1).
\end{align*}
Finally, $2\theta(s+1) \leqslant n_1 \theta(s)$ for some $n_1$, it remains to set $n = \max(n_0,n_1)$.
\end{proof}

\begin{definition}\label{defn:sim}
Let $\alpha, \beta$ be quasi-geodesic rays or $(L, \theta)$-rays for some $L$ and $\theta$ in a proper geodesic space. We say $\alpha \sim \beta$ if either of the following holds:
\begin{enumerate}[(1)]
\item $\lim_{t \to \infty} \frac{d(\alpha(t), \beta)}{t} =0.$ \label{item:dalphatbeta}
\item $\lim_{t \to \infty} \frac{d(\beta(t), \alpha)}{t} =0.$ \label{item:dbetatalpha}
\end{enumerate}
\end{definition}

\begin{lemma}\label{lem:symmetry-basic} \eqref{item:dalphatbeta} and \eqref{item:dbetatalpha} are equivalent.
\end{lemma}

Beware that Lemma \ref{lem:symmetry-basic}  is false if one replaces $\alpha$ and $\beta$ with arbitrary maps, even proper maps from the half-line to $X$ that respect the inequality on the right hand side of \eqref{eq:kappa-ray}. For instance, if one considers the plane parametric curve $\beta$ parametrized by arc-length progressing along the horizontal axis and making jumps of height $2^n$ at time $2^n$ for all $n \geqslant 0$, and the parametrization of the horizontal axis $\alpha$, then this pair has \eqref{item:dalphatbeta} but not \eqref{item:dbetatalpha}. 

\begin{proof}
Assume \eqref{item:dalphatbeta} holds and define
$\eta(s) = d(\alpha(s), \beta)$ for $s \in[0,+\infty)$; this function $\eta$ is sublinear.
Let us introduce the set 
\[ \mathcal T = \left\{ t \in [0, + \infty) : \exists s \in [0, + \infty), d(\alpha(s), \beta(t)) \leqslant 2 d(\alpha (s), \beta) \right\}. \]

We claim that $\mathcal T$ cannot have large holes, more precisely there is no $\varepsilon >0$ and sequence $(t_n)$ with limit $+\infty$ such that
\begin{equation}
(t_n(1-\varepsilon); t_n(1+\varepsilon)) \cap \mathcal T = \emptyset 
\label{eq:to-contradict-T}
\end{equation}
for all $n$.
Indeed, assume the contrary and define 
\[ s_n = \sup \{ s : \exists t < t_n, d(\alpha(s), \beta(t)) \leqslant 2 d(\alpha(s), \beta) \} \]
It can be checked that $s_n$ is well defined for all $n$ and tends to $+\infty$, since $\alpha$ is proper while $\beta[0,t_n]$ is bounded and $X$ is proper.
Now consider a nearest point projection $p_n$ of $\alpha(s_n+1)$ on $\beta$.
Necessarily, $t > t_n(1+\varepsilon)$ for every $t$ such that $p_n= \beta(t)$, in view of the definition of $s_n$ and $t_n$.
Let $q_n$ be a nearest-point projection of $\alpha(s_n)$ on $\beta$; by the same argument, if $q_n = \beta(t)$ then $t<t_n(1-\varepsilon)$.
Now, on the one hand by the triangle inequality,
\begin{align}
d(p_n, q_n) & \leqslant d(p_n, \alpha(s_{n}+1) + d(\alpha (s_n+1), \alpha (s_n)) + d(\alpha(s_n), q_n)) \notag \\
& \leqslant L + \theta(s_{n+1}) + \eta(s_n) + \eta(s_{n} + 1). \label{eq:first-ineq-pn-qn}
\end{align}
On the other hand, by the left-hand side of \eqref{eq:kappa-ray},
\begin{align}
d(p_n, q_n) & \geqslant 2L^{-1} \varepsilon t_n - \widehat \theta (\max( \Vert p_n \Vert, \Vert q_n \Vert)) 
 \label{eq:second-ineq-pn-qn}
\end{align}
where $\widehat \theta = O(\theta)$.
However, there is a constant $M >0$ such that $M^{-1} t_n \leqslant s_n \leqslant Mt_n$ for $n$ large enough; one can take $M = 2L^2(1+\varepsilon)$, by considering the first inequality in Remark \ref{rem:linear-control}.

Making $n \to + \infty$ and using that $\theta$, $\widehat \theta$ and $\eta$ are sublinear, \eqref{eq:first-ineq-pn-qn} and \eqref{eq:second-ineq-pn-qn} are in contradiction with one another.
Thus \eqref{eq:to-contradict-T} cannot be true.
It follows that there is a sublinear function $\mu$ such that for all $t \geqslant 0$, there is $t'$ with $\vert t - t' \vert \leqslant \mu(t)$ and $t' \in \mathcal T$. 
And then $d(\beta(t), \alpha) \leqslant d(\beta(t), \beta(t')) + d(\beta(t'), \alpha)$, which is bounded above by a sublinear function of $t$ involving $\eta$, $\mu$ and $L$ in view of the definition of $\mathcal T$ and the linear control between $s$ and $t$ when $\beta(t')$ is a closest point projection of $\alpha(s)$ on $\beta$.
We proved that \eqref{item:dalphatbeta} implies \eqref{item:dbetatalpha} for $(L, \theta)$ rays; the converse implication holds by symmetry, and quasi-geodesics are $\theta$-rays, hence the Lemma is proved.  \qedhere
\end{proof}
\begin{figure}

\begin{tikzpicture}[line cap=round,line join=round,>=angle 45,x=0.5cm,y=0.5cm]
\clip(-3,-1) rectangle (12,5);

\draw [color=black, line width=1pt,domain=-2.2:10, samples = 50, variable = \t] plot({\t},{0.1*sin(\t r)}) node[right]{$\alpha$};

\draw [color=black, line width=1pt,domain=-2:10.5, samples = 500, variable = \t] plot({\t- 3/(1+(\t-6)*(\t-6)) -  3/(1+(\t-8)*(\t-8))},{4/(1+(\t-7)*(\t-7))+ 1/(1+\t*\t)}) node[right]{$\beta$};

\fill (0,1.1) circle (1.5pt) node[above] {$\beta(t_{n-1})$};
\fill (4,4) circle (1.5pt) node[above] {$\beta(t_n)$};

\fill (3.5,-0.04) circle (1.5pt) node[below left] {$\alpha(s_n)$};
\fill (5,-0.1) circle (1.5pt) node[below right] {$\alpha(s_n +1)$};

\fill (5.5,1.3) circle (1.5pt) node[above right] {$p_n$};
\fill (3.3,0.5) circle (1.5pt) node[above left] {$q_n$};

\fill (4.5,1.9) circle (1.5pt) node[above right] {$\beta(t_{n}(1+\varepsilon))$};
\fill (2.5,2.2) circle (1.5pt) node[above left] {$\beta(t_n(1-\varepsilon))$};

\draw (5,-0.1) -- (5.5, 1.3);
\draw [->] (5.4,1.02) -- (5.7, 0.9) -- (5.8,1.25);

\draw (3.5,-0.04) -- (3.3, 0.5);

\end{tikzpicture}
\caption{Proof of Lemma \ref{lem:symmetry-basic}.}
\end{figure}
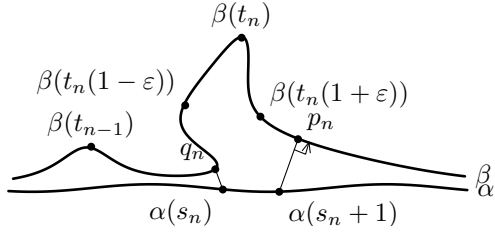

\subsection{Definition of the sublinearly Morse boundary $\partial X$}
\label{subsec:dfn-pka-boundaries}

For more extensive references on sublinearly Morse boundaries, see \cite{QRT19} and \cite{QRT20}. A $(\mathsf q,\mathsf Q)$-quasigeodesic is a map $\gamma : [0,+\infty) \to X$ such that 
$ \frac{1}{\mathsf q}d(x,y) - \mathsf Q \leqslant d(\gamma(x), \gamma(y)) \leqslant \mathsf q d(x,y) + \mathsf Q $
for all $x,y \in [0, + \infty)$.
 
\begin{definition}[$\kappa$-Morse quasigeodesic {\cite[Definition 3.2]{QRT20}}]\label{defn:Morse}
Let $Z \subseteq X$ be a closed subspace. Let $m_Z: \mathbb R^2 \to \mathbb R$ be a proper function.
We say that $Z$ is $\kappa$-Morse with Morse gauge $m_Z$ if for every sublinear function $\kappa'$, for every $r>0, n>0$ such that $m_Z(\mathsf q,\mathsf Q) \ll_{\kappa} r $, there exists $R(Z, r, n, \kappa')>0$ such that if $\beta$ is a $(\mathsf q, \mathsf Q)$-quasigeodesic then
\[
d_X(\beta_R, Z) \leqslant n \cdot \kappa'(R) \implies \beta|_r \subset \mathcal{N}_{\kappa} (Z, m_Z(\mathsf q, \mathsf Q)).
\]
\end{definition}

Applying the definition to all $r>0$ we have the following alternative characterization:
\begin{definition}\label{defn2:Morse}
Let $Z$ be a closed subspace. Let $m_Z: \mathbb R^2 \to \mathbb R$ be a proper function.
We say that $Z$ is $\kappa$-Morse with Morse gauge $m_Z$ if $\beta$ is any other $(\mathsf q, \mathsf Q)$-quasi-geodesic ray that sublinearly tracks $Z$, then we have:
\[
\beta \in  \mathcal{N}_{\kappa} (Z, m_Z(\mathsf q, \mathsf Q)).
\]
\end{definition}

\begin{proposition}[{\cite[Lemma 3.4 and Corollary 3.5]{QRT20}}]\label{straightandnot}
Let $\alpha$ and $\beta$ are quasi-gedodesic rays in $X$,  such that $\beta$ is $(\mathsf q, \mathsf Q)$-quasi-geodesic and $\alpha$ is $\kappa$-Morse. 
If $\alpha \sim \beta$ then $\beta$ is $\kappa$-Morse with gauge
$\mm_\alpha + 4 \mm_\alpha(\mathsf q, \mathsf Q)$.
\end{proposition}

We will reprove this in a greater generality in Proposition~\ref{closeisMorse}.


Define $\partial_\kappa X$ as the set of $\kappa$-Morse quasigeodesic up to $\sim$. For any $\kappa$-Morse $\beta$ , define $\partial \mathcal U(\beta,r)$ as
\[ \lbrace \mathbf a: \alpha\in \mathbf a \text{ is a $(\mathsf q, \mathsf Q)$-quasigeodesic and } r \gg_{\kappa} \mm_\beta(\mathsf q, \mathsf Q) \implies  \alpha_{\mid r} \subseteq \mathcal N_{\kappa}(\beta, \mm_\beta(\mathsf q, \mathsf Q)) \rbrace, \]
and then
\[ \partial \mathcal B(\mathbf b) = \lbrace \mathcal V \subseteq \partial_\kappa X : \exists \beta \in \mathbf b, \exists r>0, \mathcal V \supseteq \partial \mathcal U(\beta, r) \rbrace.
\]
The topology on $\partial_\kappa X$ is defined as the unique one so that the $\partial \mathcal B(\mathbf b)$ are the neighbourhood systems at $\mathbf b$ \cite[Lemma 4.5]{QRT20}; it is metrizable \cite[Lemma 4.8]{QRT20}.
Finally, the {sublinearly Morse boundary} is defined as $\partial X =\cup_\kappa \uparrow \partial_ \kappa X$.

\subsection{Sublinear biLipschitz equivalence}

In this paragraph, $\theta$ is a sublinear function with the same properties as $\kappa$.

\begin{definition}
Let $Z$ and $Z'$ be two closed unbounded subsets in $X$.
Say that $Z$ and $Z'$ linearly separate if $d(Z \cap S(\mathfrak o, r), Z') -r$ stays bounded as $r \to + \infty$.
\end{definition}

\begin{definition}[$\theta$-SBE]\label{defn:theta-SBE}
Let $(X, \mathfrak o)$ and $(Y, \mathfrak o)$ be proper geodesic metric spaces with basepoints. Let $L \geqslant 1$ be a constant, and let $\theta$ be a sublinear function as before.
We say that $\Phi : X \to Y$ is a $(L, \theta)$-sublinear biLipschitz equivalence ($\theta$-SBE for short) if
\[ \frac{1}{L}d(x_1,x_2)  - \theta(\max (\Vert x_1 \Vert, \Vert x_2 \Vert)) \leqslant d(\Phi(x_1), \Phi(x_2))) \leqslant Ld(x_1,x_2)  +\theta (\max (\Vert x_1 \Vert, \Vert x_2 \Vert))\]
and $Y = \mathcal{N}_\theta(\Phi(X), D)$ for some $D \geqslant 0$.
\end{definition}

\begin{proposition}[Inverses]\label{prop:inverses-SBE}
Let $\Phi: X \to Y$ be a $\theta$-SBE. Then there exists $\overline \Phi : Y \to X$ a $\theta$-SBE and $n >0$ such that for all $x \in X$,
\begin{equation}
\label{eq:first-prop-inverse-SBE}
d(x,\overline{\Phi}(\Phi(x))) \leqslant n\cdot \theta(x )
\end{equation}
and for all $y \in Y$, 
\begin{equation}
\label{eq:second-prop-inverse-SBE}
d(y,{\Phi}(\overline \Phi(y))) \leqslant n \cdot \theta(y ).
\end{equation}
\end{proposition}

\begin{proof}
This follows from \cite[Proposition 2.4]{cornulier2017sublinear}, however, let us give here a self-contained proof. For simplicity, let us assume for now that $\Phi(X)$ is closed in $Y$; we will see in the end how to remove this assumption if necessary.
For every $y \in Y$, define $\overline \Phi(y)$ as some $x \in X$ such that $\Phi(X)$ is a nearest-point projection of $y$ on $\Phi(X)$. 
By assumption, $\Phi$ is a $\theta$-SBE, hence applying the last property in Definition \ref{defn:theta-SBE} one gets that 
\begin{equation}
\label{eq:almost-second-prop-inverse-SBE}
d(y, \Phi(\overline \Phi(y))) \leqslant D \theta(\Phi(x)),
\end{equation}
for some $D \geqslant 0$.
This is almost \eqref{eq:second-prop-inverse-SBE}; let us rework this inequality slightly.
Since we know that $d(y, \Phi(x)) \leqslant D \theta(\Vert y \Vert)$, when $y$ is far enough from $\mathfrak o$, there is some constant $K$ so that as soon as $\Vert y \Vert \geqslant K$, $\Vert y \Vert / 2 \leqslant \Vert \Phi(x) \Vert \leqslant 2 \Vert y \Vert$.
It follows that for some constant $K'$, for all $y$, $\Vert \Phi(x) \Vert \leqslant 2 \Vert y \Vert + K'$.
Hence
\[ d(y,{\Phi}(\overline \Phi(y))) \leqslant  D \theta( \Phi(x) ) \leqslant D \theta(2 \Vert y \Vert + K') = O(\theta( y )).\]
Thus we proved \eqref{eq:second-prop-inverse-SBE}.
Now for \eqref{eq:first-prop-inverse-SBE}, note that $x$ and $\overline \Phi (\Phi(x))$ have the same image, namely $\Phi(x)$, through $\Phi$.
So $\frac{1}{L}d(x, \overline \Phi(\Phi(x)) - \theta(\max (\Vert x \Vert, \Vert \overline \Phi(\Phi(x)) \Vert) \leqslant d(\Phi(x), \Phi(x)) =  0$, whence
\begin{equation}
\label{eq:d(x,line-Phi-Phix)-inter}
d(x, \overline \Phi(\Phi(x)) \leqslant L \theta (\max (\Vert x \Vert, \Vert \overline \Phi(\Phi(x)) \Vert). 
\end{equation}
But also $\frac{1}{L}\Vert \overline \Phi(\Phi(x) \Vert - \theta(\overline \Phi(\Phi(x))) \leqslant \Vert \Phi(x) \Vert \leqslant L \Vert x \Vert + \theta ( x )$. There exists $K$ such that $\theta(r) \leqslant K + r/(2L))$, and then
\[ \frac{1}{2L}\Vert \overline \Phi(\Phi(x) \Vert   \leqslant L \Vert x \Vert + \frac{\Vert x \Vert}{2L} + 2K. \]
Plugging this into \eqref{eq:d(x,line-Phi-Phix)-inter},
\[ d(x, \overline \Phi(\Phi(x)) \leqslant L \theta (\max (\Vert x \Vert, (2L^2+1)\Vert x \Vert +2LK) = O(\theta(\Vert x \Vert)). \]
Finally we need to prove that $\overline \Phi$ is a $\theta$-SBE.
Applying the inequality on the right in Definition \ref{defn:theta-SBE} for $\Phi$, 
\begin{align*}
d(\overline \Phi(y), \overline \Phi(y')) & \leqslant L d(\Phi \overline \Phi(y), \Phi \overline \Phi(y')) + \theta (\max (\Vert \overline \Phi(y) \Vert, \Vert \overline \Phi(y') \Vert)) \\
& \leqslant L d(y,y') + 2 \theta  (\max (\Vert y \Vert, \Vert y' \Vert, \Vert \overline \Phi(y) \Vert, \Vert \overline \Phi(y') \Vert))
\end{align*}
Note that 
\[ \frac{1}{L}\Vert \overline \Phi (y) \Vert - \theta (\Vert \overline \Phi(y) \Vert) \leqslant d(\Phi \overline \Phi(y), \Phi(\mathfrak o))
\leqslant \Vert y \Vert + O(\theta(\Vert y \Vert)). \]
so that $\Vert \overline \Phi (y) \Vert \leqslant 2L \Vert y \Vert + M$ for some constant $M$. Thus
\[ d(\overline \Phi(y), \overline \Phi(y')) \leqslant L d(y,y') + 2 \theta  (\max (\Vert y \Vert, \Vert y' \Vert, \Vert \overline \Phi(y) \Vert, \Vert \overline \Phi(y') \Vert)) \leqslant L d(y,y') + 2P \theta  (\max (\Vert y \Vert, \Vert y' \Vert) \]
for some $P >0$.
This proves the inequality on the right in Definition \ref{defn:theta-SBE} for $\overline \Phi$ ; in the exact same way, the left inequality on the left is obtained by using the inequality on the left for $\Phi$.

It remains to check that $X = \mathcal{N}_\theta(\overline \Phi(Y), \overline  D)$ for some $\overline D \geqslant 0$. This follows from \eqref{eq:first-prop-inverse-SBE} exactly the same way that we deduced \eqref{eq:second-prop-inverse-SBE} from \eqref{eq:almost-second-prop-inverse-SBE}.

Finally, $\Phi(X)$ may not be closed in $Y$, but in the construction of $\overline \Phi(y)$, we can relax the condition defining $\overline \Phi(y)$ by replacing the nearest-point projection of $y$ with some point at distance at most $2 d(y, \Phi(X))$ from $y$. All the estimates afterwards go through with additional multiplicative constants.
\end{proof}

\begin{remark}
When $\theta = 1$, the proof is easier and it is one of the first exercises on quasiisometries in textbooks; see e.g. \cite[Exercise 8.12]{DrutuKapovich}.
\end{remark}

\begin{lemma}
Let $\alpha$ be an $(L, \theta)$-ray. 
Let $\Phi$ and $\overline \Phi$ be as in Proposition \ref{prop:inverses-SBE}.
Let $\widehat \alpha$ and $\widehat {\overline \Phi \Phi \alpha}$ be continuous completions of $\alpha$ and $\overline \Phi \Phi \alpha$.
Then there exists $n$ depending on $L$ and $\theta$ and $n'$ depending on $L,\theta$ and $\Phi$, such that $\widehat \alpha \subset \mathcal N_\theta(\widehat{\overline \Phi \Phi \alpha}, n)$ and  $\widehat{\overline \Phi \Phi \alpha} \subset \calN_\theta(\widehat \alpha, n')$.
\end{lemma}

\begin{proof}
By Proposition \ref{prop:inverses-SBE}, there exists $n_0$ such that for every $x$ in $\alpha$, $d(x, \overline \Phi \Phi x) \leqslant n_0 \theta (x)$. Hence $\alpha \subset \mathcal N_\theta(\overline \Phi \Phi \alpha, n_0)$. Morever, by Equation~\eqref{eq:completion-is-close} and Remark \ref{rem:linear-control}, there is $n_1$ depending on $L$ and $\theta$ such that
\[ d(\widehat \gamma(t), \gamma(t) \leqslant n_1 \theta (\max(\Vert \gamma(t) \Vert, \Vert \widehat \gamma(t) \Vert) \]
for all $t$. It follows that 
\[ \alpha \subset \mathcal N_\theta (\overline \Phi \Phi \alpha, n_0 + n_1). \]
By Lemma \ref{Lem:sublinear-estimate}, since $d(x, \overline \Phi \Phi x) \leqslant n_0 \kappa(x)$ for all $x$ in $\alpha$, there are $\mathsf D_1$, $\mathsf D_2$ such that $\mathsf D_1 \kappa(x) \leqslant \kappa (\overline \Phi \Phi x) \leqslant \mathsf D_2 \kappa(x)$ for every $x$ on $\alpha$.
Finally, $\overline \Phi \Phi \alpha$ is a $(L', \theta)$ ray where $L'$ depends on $L$ and $\Phi$. So applying again Equation~\eqref{eq:completion-is-close} and Remark \ref{rem:linear-control}, there exists $n_1'$ such that 
\[ d(\widehat {\overline \Phi \Phi \gamma}(t), \overline \Phi \Phi \gamma(t) ) \leqslant n_1 \theta (\max(\Vert \overline \Phi \Phi \gamma(t) \Vert, \Vert \widehat {\overline \Phi \Phi \gamma}(t) \Vert). \]
Setting $n' = n_0 + n_1'$ finishes the proof.
\end{proof}

There exist several degree of closeness between $\kappa$-rays.
The first is the $\sim$ relation defined earlier.



\begin{definition}[$\kappa$-fellow travelling rays]
Given two rays $\alpha$ and $\beta$ (which are frequently either quasi-geodesic rays or $\theta$-rays in this paper) we say, $\alpha$ and $\beta$ \emph{$\kappa$-fellow travel}  each other if there exists $n$ and for all $t>0$, we have
\[
d(\alpha(t), \beta(t)) \leq n \cdot \kappa(t).
\] 
Further, we say that $\alpha$ and $\beta$ $\kappa$-track each other
if there exists $n_1$ such that 
\[
d(\alpha_r, \beta_r) \leq n_1 \cdot \kappa(r).
\] 
for all $r$.
\end{definition}

Note that if $\alpha$ and $\beta$ $\kappa$-fellow travel each other, then they $\kappa$-track each other, however the converse is not true.

\begin{lemma}
\label{lem:kappa-Morse-transfers-to-fellow-travelling}
Let $\alpha$ and $\alpha'$ be fellow-travelling $(L, \theta)$-rays, and assume that $\alpha$ is $\kappa$-Morse.
Then, $\alpha'$ is $2\kappa + \theta$-Morse.
\end{lemma}

\begin{proof}
%
%
%

We will show that the characterization of Definition~\ref{defn2:Morse} holds here.
Let $\beta$ be quasi-geodesic ray that sublinearly tracks $\alpha'$ and let that sublinear tracking function be $\rho(t)$. 
Since $\alpha'$ and $\alpha$
 are $\kappa$-fellow traveling $\theta$-rays, we have that $\beta$ is in a $(\rho+\kappa+\theta)$-neighbourhood of $\alpha$.  Since $\alpha$ is $\kappa$-Morse, by Definition~\ref{defn2:Morse}
 we have that $\beta$ is in a $\kappa$-neighbourhood of $\alpha$. But since $\alpha$ is in a $\kappa+\theta$ neighbourhood of $\alpha'$, we have that $\beta$ is in a $2\kappa+\theta$  neighbourhood of $\alpha'$. Therefore, $\alpha'$ is $2\kappa+\theta$-Morse.
\end{proof}

\section{$\theta$-SBE invariance of the $\kappa$-Morse quasi-geodesic rays, when $\theta = O(\kappa)$}
\label{sec:SBE-Morse-ray-to-Morse-ray}
In this section we establish what happens to $\kappa$-Morse quasi-geodesic rays under $(L, \theta)$ maps. We prove that they behave well in the sense that they are sent to images sublinearly tracking sublinearly Morse geodesic rays, provided that $\kappa$ dominates $\theta$.

Precisely, given two sublinear functions $\kappa, \theta$, we say $\kappa$ \emph{dominates} $\theta$ if there exists constants $C_1, C_2$ and some $t_0$ such that for all $t>t_0$,
\[ \kappa(t) \geq C_1 \theta(t) + C_2.\]
Therefore in the results proven, we show frequently a ray is $(\kappa+\theta)$-Morse, which implies it is $\kappa$-Morse if $\theta \preceq \kappa$ and it is $\theta$-Morse if $\kappa \preceq \theta$.

\subsection{$\kappa$-Morse rays} 

Assume that $\alpha$ is a $\kappa$-Morse $(L, \theta)$-ray in a proper metric space. Then we can establish that a quasi-geodesic ray that tracks $\alpha$ sublinearly is itself $\kappa$-Morse.

\begin{proposition} \label{closeisMorse}
Assume that $\kappa$ dominates $\theta$.
Let $\alpha$ be an $(L, \theta)$-ray that is  $\kappa$-Morse, and let $\beta \sim \alpha$ be a $(\qq, \sQ)$-quasigeodesic ray that tracks $\alpha$ sublinearly. Then $\beta$ is $\kappa$-Morse.
\end{proposition}

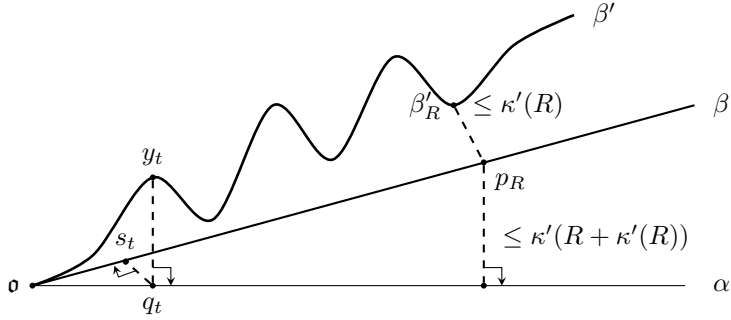
\begin{figure}[h!]
\begin{tikzpicture}[scale=0.8]
 \tikzstyle{vertex} =[circle,draw,fill=black,thick, inner sep=0pt,minimum size=.5 mm]
[thick, 
    scale=1,
    vertex/.style={circle,draw,fill=black,thick,
                   inner sep=0pt,minimum size= .5 mm},
                  
      trans/.style={thick,->, shorten >=6pt,shorten <=6pt,>=stealth},
   ]

  \node[vertex] (o) at (0,0)[label=left:$\go$] {}; 
  \node(a) at (11,0)[label=right:$\alpha$] {}; 
  
  \draw (o)--(a){};
  \draw [-, thick] (o) to (11, 3){};
  \node at (11,3) [label=right:$\beta$] {}; 
  \node at (9,4.5) [label=right:$\beta'$] {};
   
  \node at (7, 3) [label=right:$\leq \kappa'(R)$] {}; 
  \node at (7.5, .8) [label=right:$\leq \kappa'(R + \kappa'(R))$] {}; 
   
  \node[vertex] (a1) at (2, 1.8)[label=above:$y_t$] {}; 
  \node[vertex] (a2) at(2, 0)[label=below:$q_{t}$] {};

  \node[vertex] (a3) at(1.55, 0.4){}; 
   \node[vertex] at(1.55, 0.4)[label=above:$s_{t}$] {}; 
  \node[vertex] (b1) at(7, 3)[label=left:$\beta'_R$] {}; 
  \node[vertex](b2) at (7.5, 2.05)[label=below right:$p_{R}$] {}; 
  \node[vertex] (b3) at(7.5, 0)[label=left:$ $] {}; 
   
  \draw[thick, dashed] (a1)to(a2) {}; 
  \draw[thick, dashed] (a2)to(a3) {}; 
   
  \draw[thick, dashed] (b1) to (b2) {}; 
  \draw[thick, dashed] (b2) to (b3) {}; 
  
  \draw [->] (2,0.3) -- (2.3,0.3) -- (2.3,0);
  \draw [->, shift={(7.5,0)}] (0,0.3) -- (0.3,0.3) -- (0.3,0);
  \draw [->, shift={(1.55,0.4)}, rotate=200] (-0.1,0.2) -- (0.2,0.2) -- (0.2,0);

  \pgfsetlinewidth{1pt}
  \pgfsetplottension{.55}
  \pgfplothandlercurveto
  \pgfplotstreamstart
  \pgfplotstreampoint{\pgfpoint{0cm}{0cm}}
  \pgfplotstreampoint{\pgfpoint{1cm}{0.5cm}}
  \pgfplotstreampoint{\pgfpoint{2cm}{1.8cm}}
  \pgfplotstreampoint{\pgfpoint{3cm}{1.1cm}}
  \pgfplotstreampoint{\pgfpoint{4cm}{3cm}}
  \pgfplotstreampoint{\pgfpoint{5cm}{2.1cm}}
  \pgfplotstreampoint{\pgfpoint{6cm}{3.8cm}}
  \pgfplotstreampoint{\pgfpoint{7cm}{3cm}}
  \pgfplotstreampoint{\pgfpoint{8cm}{4cm}}
  \pgfplotstreampoint{\pgfpoint{9cm}{4.5cm}}
  \pgfplotstreamend
  \pgfusepath{stroke} 
  \end{tikzpicture}
  \caption{The setup in the proof of Proposition~\ref{closeisMorse}.}
\end{figure}


\begin{proof}
The proof is similar to that of \cite[Lemma 3.4]{QRT20} where $\beta$ is also $(\mathsf q, \mathsf Q)$-quasi-geodesic but $\alpha$ is a $\kappa$-Morse quasi-geodesic instead of a $\kappa$-Morse $(L,\theta)$-ray.
First, assume that $\alpha$ is continuous. By an identical proof of \cite[Lemma III.1.11]{BH99}, an $(L, \theta)$-ray $\alpha$ can always be made to be continuous 
as an $(L, 2\theta)$-ray, thus we assume without loss of generality from now on that $\alpha$ is continuous.
 
 Define $\kappa'(r) := d_X(\beta_r, \alpha)$.
 By Definition~\ref{defn:sim} and Lemma~\ref{lem:symmetry-basic}, the function $\kappa'$ is sublinear.
Let us now prove that $\beta$ is $\kappa$-Morse. 
Let $r > 0$ and let $\beta'$ be a $(\qq', \sQ')$-quasi-geodesic ray such that 
\[ d_X(\beta'_R, \beta) \leq \kappa'(R) \]
for some sufficiently large $R$. 
Let $p_R$ be a nearest point projection 
of $\beta'_R$ to $\beta$; by construction and by triangle inequality, we have 
$$\Vert p_R \Vert \leq 2 \Vert \beta_R' \Vert = 2 R.$$
Then, by the triangle inequality, 
\begin{align*}
d_X(\beta'_R, \alpha) & \leq d_X(\beta'_R, p_R) + d_X(p_R, \alpha) \\
& \leq  \kappa'(R) + \mm_\alpha(\qq, \sQ)\cdot \kappa (\Vert p_R \Vert) \, \, \, \,  \text{$\beta$ is a quasi-geodesic ray and hence}\\
& \hspace{5cm} \text{is in a $\kappa$-neighbourhood of $\alpha$.}\\
& \leq \kappa'(R) + \mm_\alpha(\qq, \sQ)\cdot \kappa(2 R).
\end{align*}
Since $\kappa''(R) := \kappa'(R) + m(\qq, \sQ)\cdot \kappa(2 R)$ is also a sublinear function, and since $\alpha$ is $\kappa$-Morse, 
this implies that 
\begin{equation}
\beta'|_{r} \subseteq \mathcal{N}_{\kappa}(\alpha, \mm_\alpha(\qq', \sQ')). 
\end{equation}

Let $y_t$ be any point on $\beta'$ with $\Vert y_t \Vert = t \leq r$. 
By construction and triangle inequality, if $q_t$ is a nearest point projection of $y_t$ to $\alpha$, we have 
 $$\Vert q_t \Vert \leq 2 \Vert y_t \Vert = 2 t. $$
Now, if $q$ is any point on $\alpha$ and $s$ is a nearest point projection of $q$ to $\beta$,  by the triangle inequality and the Morse property,
$$\Vert q \Vert \geq \Vert s \Vert - d_X(s, q) \geq  \Vert s \Vert - \mm_\alpha(\qq, \sQ)\cdot \kappa(\Vert s \Vert).$$
Moreover, again by construction and triangle inequality, $\Vert s \Vert \leq 2 \Vert q \Vert$,
hence by concavity
$$d_X(s, q) \leq \mm_\alpha(\qq, \sQ)\cdot \kappa(\Vert s \Vert) \leq  2 \mm_\alpha(\qq, \sQ)\cdot \kappa (\Vert q \Vert).$$
Thus, let $s_t$ be a nearest point projection of $q_t$ to $\beta$, the above estimate yields
\begin{align*}
d_X(q_t, s_t) & \leq 2 \mm_\alpha(\qq, \sQ)\cdot \kappa (\Vert q_t \Vert) \\
& \leq 4 \mm_\alpha(\qq, \sQ)\cdot \kappa (t)
\end{align*}
hence, putting everything together, 
\begin{align*}
d_X(y_t, \beta) & \leq d_X(y_t, q_t) + d_X(q_t, s_t) \\
& \leq \mm_\alpha(\qq', \sQ')\cdot \kappa(t) +  4 \mm_\alpha(\qq, \sQ)\cdot \kappa (t)
\end{align*}
which, by setting $\mm_\beta(\qq', \sQ') := \mm_\alpha(\qq', \sQ')+ 4 \mm_\alpha(\qq, \sQ)$, 
proves the claim.
Finally, $\alpha$ may not be continuous.
In the case where it is not, using the notation from Lemma \ref{lem:connect-dots}, we have to check that $\widehat \alpha$ inherits the $\kappa$-Morse property, which follows readily from Lemma \ref{lem:kappa-Morse-transfers-to-fellow-travelling}.
\end{proof}

Building on these results we are now ready to establish the map on $\pka X$ that is induced by an SBE.
\begin{lemma}\label{lemma1}
If $\alpha$ is a  (\qq, \sQ)-$\kappa$-Morse quasi-geodesic ray in $X$, and $\Phi$ is an $(L, \theta)$-sublinear bi-Lipschitz equivalence between two proper metric spaces $X$ and $Y$.  Then $\overline \Phi\Phi \alpha$ is $(\kappa+\theta)$-Morse.
\end{lemma}
\begin{proof}
Consider $\alpha, \overline \Phi\Phi \alpha$. By Proposition~\ref{prop:inverses-SBE},  every point $x$ of $\alpha$ is $\theta(x)$ away from the point $\overline \Phi\Phi (x)$.  
Consider a $(\qq, \sQ)$-quasi-geodesic $\gamma$ that is sublinearly tracking 
$\overline\Phi \Phi \alpha$. 
Let this tracking function be denoted $\kappa'$. 

Then, there exists a constant $C$ such that any point  $y$ on $\gamma$ is distance at most $C(\kappa'(y)+\theta(y))$ from $\alpha$. 
Indeed, $d(y, \overline \Phi \Phi \alpha_{\Vert y \Vert}) \leqslant \kappa'(y)$ and then by Proposition \ref{prop:inverses-SBE}, $d(y, \alpha) \leqslant \kappa'(y) + \widehat \theta (y)$ where $\widehat{\theta} \leqslant C \theta$ for some $C$ by \eqref{eq:kappa-ray-reformulated}.

Since $\alpha$ is $\kappa$-Morse and $C(\kappa'(y)+\theta(y))$ is a sublinear function, by Proposition \ref{closeisMorse} we have that $\gamma$ is in a $\kappa$-neighbourhod of $\alpha$ and $\gamma$ is $\kappa$-Morse.

Therefore we have that $\gamma$ is in $\calN_\kappa(\alpha, m)$ for some $m$. Since $\alpha$ is $\theta$-tracking  $\overline \Phi\Phi \alpha$, we have that $\gamma$ is $m \cdot \kappa+\theta$ tracking  $\overline \Phi\Phi \alpha$. Thus we have shown that any quasi-geodesic ray that sublinearly tracks $\overline \Phi\Phi \alpha$ with rate $m \cdot \kappa+\theta$. Thus  $\overline \Phi\Phi \alpha$ is $(\kappa+\theta)$-Morse.  
\end{proof}

\begin{figure}[h!]
\begin{tikzpicture}[scale=0.8]
\draw [color=black, line width=1pt,domain=0:5, samples = 100] plot({\x,0.5*\x+0.1*\x*sin(10*\x r)}) node[anchor=south]{$\alpha$};
\draw [dash pattern= on 3pt off 2pt, color=black, line width=1pt,domain=0:5, samples = 100] plot({\x,0.1*\x*sin(7*\x r)}) node[anchor=west]{$\gamma$};

%
%
%

\draw [color=black, line width=1pt,domain=0:5.5, samples = 100, variable=\t] plot({\t-0.1*\t*cos(5*\t r)},{0.2*\t+(0.08*\t+0.01*\t*\t)*sin(5*\t r)}) node[anchor=west]{$\overline \Phi \Phi \alpha$};

%
%
%
%

\end{tikzpicture}
\caption{Proof of Lemma~\ref{lemma1}.}
\end{figure}

%
%
%
%
%
%
%
%
%
%
%
%
%
%
%
%

\begin{proposition}\label{imageMorse}
If $\alpha$ is a  (\qq, \sQ) $\kappa$-Morse quasi-geodesic ray in $X$, and $\Phi$ is an $(L, \theta)$-sublinear bi-Lipschitz equivalence between two proper metric spaces $X$ and $Y$. Then $\Phi\alpha$ is $(\kappa+\theta)$-Morse.
\end{proposition}
\begin{proof}
Consider a $(\qq, \sQ)$-quasi-geodesic ray $\gamma$ that sublinearly tracks $\Phi\alpha$ and let the tracking function be $\kappa'$. Consider the image  $\overline \Phi\gamma$. $\overline \Phi\gamma$ is distance $\theta + \kappa'$ from $\overline \Phi\Phi \alpha$.  Since $\overline \Phi\Phi \alpha$ is $\theta$-tracking $\alpha$ we have that $\overline \Phi\gamma$ is 
$2\theta + \kappa'$-tracking $\alpha$. By Proposition~\ref{closeisMorse},  $\overline \Phi\gamma$ is $\kappa$-Morse.  Therefore, since $\alpha$ is a quasi-geodesic that sublinearly tracks  $\overline \Phi\gamma$, we have that $\alpha$ and  $\overline \Phi\gamma$ are $\kappa$-close.

Now apply $\Phi$ to both $\alpha$ and  $\overline \Phi\gamma$. $\Phi \alpha$ and  $\Phi \overline \Phi\gamma$ are $\kappa+ \theta$ apart. Since $\Phi \overline \Phi\gamma$ and $\gamma$ are $\theta$ apart, then we have that $\Phi \alpha$
and $\gamma$ are at most $\kappa+ 2\theta$ apart. This holds for every $\gamma$ and thus we have that $\Phi \alpha$ is $(\kappa+\theta)$-Morse. \qedhere
\end{proof} 

These two results guarantee that under a sublinear biLipschitz equivalence, the image of a $\kappa$-Morse quasi-geodesic ray is an $(L, \theta)$-ray that carries the $\kappa$-Morse property, i.e. any quasi-geodesic ray that sublinearly tracks this set tracks it with a uniformly controlled sublinear function. The next two results shows that the latter set, i.e. a $\kappa$-Morse  $(L, \theta)$-ray  still has strong association with a geodesic ray and thus can be connected with a $\kappa$-Morse equivalence class in $\pka X$ in Section~\ref{sectionboundary}.

\begin{lemma}\label{uniquegeodesic}
Let $\alpha$ be a $\kappa$-Morse $(L, \theta)$-ray. Then there exists a geodesic ray $\underline{a}$ such that $\underline{a}$ and $\alpha$ $\kappa$-track each other.
\end{lemma}

\begin{figure}[h!]
\begin{tikzpicture}[scale=0.8]
\draw [color=black, line width=1pt,domain=0:5.5, samples = 100, variable=\t] plot({\t-0.1*\t*cos(5*\t r)},{0.2*\t+(0.08*\t+0.01*\t*\t)*sin(5*\t r)}) node[anchor=west]{$\alpha$};

\draw [color=black, line width=0.5pt,domain=0:5.5, samples = 100, variable=\t] plot({\t},{0.01*\t*\t}) node[anchor=west]{$\underline a$};

\fill [color=black] (2.2,0.2) circle(1.5pt) node[ below] {\small $\alpha(i_1)$};

\fill [color=black] (3.5,0.3) circle(1.5pt) node[ below] {\small $\alpha(i_2)$};

\fill [color=black] (4.8,0.4) circle(1.5pt) node[ below] {\small $\alpha(i_3)$};

\end{tikzpicture}
\caption{Proof of Lemma~\ref{uniquegeodesic}.}
\end{figure}
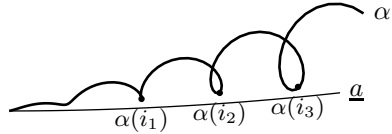

\begin{proof}
Since $\alpha$ is $\kappa$-Morse, let $\mm_\alpha(\qq, \sQ)$ be its $\kappa$-Morse gauges. 
Consider the geodesic segments connecting the base-point and $\alpha(i)$, $i=1,2,3\ldots$. 

Since $\alpha$ is $\kappa$-Morse, for every $r>0$, there exists $i_r$ such that
\[
\alpha(i_r) \in \calN_1(\alpha, 1) \, \Longrightarrow \, [\go, \alpha(r)] \in  \calN_\kappa(\alpha, m(1,0)).
\]
The space $X$ is assumed to be proper. 
Thus by the Arzel\'a-Ascoli Theorem, for some extraction $\{ i_n \}$, the subsequence $\{ [\go, \alpha(i_n)], n=1,2,3\ldots \}$ converges to a geodesic ray $\underline{a}$. 

For each $n$, all but the first $n-1$ segments in this sequence has the property that the intersection of them with the ball of radius $i_n$ is in the neighbourhood $\calN_\kappa(\alpha, m(1,0))$, thus their limit $\underline{a}$ is in the neighbourhood  $\calN_\kappa(\alpha, m(1,0))$.  
Thus $\underline{a}$ and $\alpha$ $\kappa$-track each other.
\end{proof}

Note that we used the $\kappa$-Morse property in an essential way.

When $X$ is Gromov hyperbolic, Lemma~\ref{uniquegeodesic} is a consequence of \cite[Lemma 3.4]{pallierHYPSBE}.

\begin{proposition}\label{Morseray}
If $\alpha$ is an $(L, \theta)$-ray that is $\kappa$-Morse, then $\alpha$ is in a $\kappa$-neighbourhood of a unique $\kappa$-Morse geodesic ray, up to sublinear tracking. Thus a $\kappa$-Morse $(L, \theta)$-ray is associated with a unique equivalence class $\bfa \in \pka X$.
\end{proposition}
\begin{proof}
By Lemma~\ref{uniquegeodesic}, there exists a geodesic ray $\underline{a}$ such that $\underline{a}$ and $\alpha$ $\kappa$-tracks each other. Then Proposition~\ref{closeisMorse} implies that $\underline{a}$ is $\kappa$-Morse. Since every $\kappa$-Morse geodesic ray is in a unique equivalence class $\bfa \in \pka X$, we get that $\alpha$ is associated with  $\bfa \in \pka X$. Suppose  $\alpha$ is in a $\kappa$-neighbourhood of another $\kappa$-Morse geodesic ray $\underline{a}'$, we have that $\underline{a}'$ and $\underline{a}$ $\kappa$-track each other and thus $\underline{a}' \in \bfa$. 
\end{proof}

\section{SBE on sublinearly Morse boundaries}\label{sectionboundary}

\subsection{Proof of Theorem \ref{th:invariant-topology-intro}}
In this section we look at the induced map of a given SBE $\Phi \colon X \to Y$ on $\pka X$. 
To begin with, we need the following basic observation about $(L, \theta)$-ray and geodesics that are in sublinear neighbourhoods of each other. The proof is identical to the analogous claim in \cite{QRT20} regarding quasi-geodesic rays and geodesic rays:

\begin{lemma}\label{L:symmetry} 
Let $\beta$ be a $(L, \theta)$-ray and $\underline{a}$ be a geodesic ray, both based at $\go \in X$. Suppose 
that 
\[ \beta \subseteq \mathcal{N}_{\kappa}(\underline{a}, m) \]
for some function $\kappa$ and some constant $m$. 
Then we also have
\[ \underline{a} \subseteq \mathcal{N}_{\kappa+\theta}(\beta , m')\]
for some constant $m'$. Especially, if $\kappa$ dominates $\theta$, then for some $m'$
\[ \underline{a} \subseteq \mathcal{N}_{\kappa}(\beta , m').\]
\end{lemma}

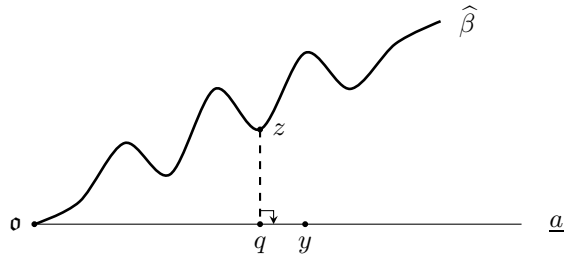
\begin{figure}[h!]
\begin{tikzpicture}[scale=0.6]
 \tikzstyle{vertex} =[circle,draw,fill=black,thick, inner sep=0pt,minimum size=.5 mm]
[thick, 
    scale=1,
    vertex/.style={circle,draw,fill=black,thick,
                   inner sep=0pt,minimum size= .5 mm},
                  
      trans/.style={thick,->, shorten >=6pt,shorten <=6pt,>=stealth},
   ]

  \node[vertex] (o) at (0,0)[label=left:$\go$] {}; 
  \node(a) at (11,0)[label=right:$\underline{a}$] {}; 
  
  \draw (o)--(a){};
  \node at (9,4.5) [label=right:$\widehat \beta$] {};
   
  \draw[thick, dashed] (5, 2.1) to(5, 0) {}; 
  \node[vertex] at(6, 0)[label=below:$y$] {}; 
  \node[vertex] at(5, 0)[label=below:$q$] {}; 
  \node[vertex] at(5, 2.1)[label=right:$z$] {}; 
  
  \draw [->] (5,0.3) -- (5.3,0.3) -- (5.3,0);

  \pgfsetlinewidth{1pt}
  \pgfsetplottension{.55}
  \pgfplothandlercurveto
  \pgfplotstreamstart
  \pgfplotstreampoint{\pgfpoint{0cm}{0cm}}
  \pgfplotstreampoint{\pgfpoint{1cm}{0.5cm}}
  \pgfplotstreampoint{\pgfpoint{2cm}{1.8cm}}
  \pgfplotstreampoint{\pgfpoint{3cm}{1.1cm}}
  \pgfplotstreampoint{\pgfpoint{4cm}{3cm}}
  \pgfplotstreampoint{\pgfpoint{5cm}{2.1cm}}
  \pgfplotstreampoint{\pgfpoint{6cm}{3.8cm}}
  \pgfplotstreampoint{\pgfpoint{7cm}{3cm}}
  \pgfplotstreampoint{\pgfpoint{8cm}{4cm}}
  \pgfplotstreampoint{\pgfpoint{9cm}{4.5cm}}
  \pgfplotstreamend
  \pgfusepath{stroke} 
\end{tikzpicture}
\caption{$\Norm y = \Norm z$ and $q \in \pi_{\underline{a}} (z)$ as in the proof of Lemma~\ref{L:symmetry}. }
\end{figure}

\begin{proof}
The core of the proof is the same as that of \cite[Lemma 3.1]{QRT20}; we reproduce it here for completeness.
First, assume for simplicity that $\beta$ is a continuous $\theta$-ray.
Let $y \in \underline{\alpha}$ be a point and let $r : = \Norm {y}$. Let $z \in \beta$ be a point such that $\Vert z \Vert = r$
and let $q$ be a nearest point projection of $z$ to $\underline{a}$. 
By assumption, 
\[
d_X(z, q) \leq m \cdot \kappa(r).
\] 
On the other hand,
\begin{align*}
d_X(y, q) & = | \Vert y \Vert- \Norm q | & \text{since $\underline{a}$ is geodesic} \\
 	   & = | \Vert z  \Vert - \Norm q | & \\
           & \leq d_X(z, q) & \text{by the triangle inequality}.
\end{align*}
Therefore we have
\[
d_X(y, \beta) \leq d_X(y, z) \leq d_X(y, q) + d_X(q, z)\leq 2 d_X(z, q) \leq 2 m \cdot \kappa(r)
\]
which completes the proof in this case. Note that we can take $m'=2m$.
Now, $\beta$ may not be continuous. Nevertheless, by Lemma \ref{lem:connect-dots} there is a continuous $\theta$-ray $\widehat \beta$ that $\theta$-fellow travels with $\beta$.
Keeping $y$ as before, and applying the previous argument to $\widehat{\beta}$, there is $z \in\widehat \beta$ with $d(y,z) \leqslant 2 m \cdot \kappa(r)$; but then also, $d(z,\beta) \leqslant m_1\theta(r)$,and then $d(y,\beta) \leqslant 2m \kappa(r)+m_1\theta(r) \leqslant m' (\theta(r) +\kappa(r))$, as required.
\end{proof}

Let $\Phi \colon X \to Y$ be a $(L, \theta)$-SBE; we are now ready to define the induced map $\Phi_\star$.   

Given a $(\mathsf q, \mathsf Q)$-quasi-geodesic ray $\zeta \from [0, \infty) \to X$ in $X$, we assume without loss of generality that the image of $\zeta$ is a continuous path. Let 
$\Phi \zeta$ be the $(L \mathsf q, \theta)$-ray in $Y$ constructed from the composition of 
$\zeta$ and $\Phi$.

\begin{proposition}\label{lem1}
Assume that $\kappa$ dominates $\theta$, and let $\Phi$ be a $(L, \theta)$-sublinear bi-Lipschitz equivalence from $X$ to $Y$, where $X$ and $Y$ are proper geodesic metric spaces. 
Two $\kappa$-Morse quasi-geodesics $\alpha$ and $\beta$ in $X$ $\kappa$-fellow travel each other if and only if $\Phi \alpha$ and $\Phi \beta$ $\kappa$-fellow travel each other in $Y$.
\end{proposition}
\begin{proof}
If $\alpha$ and $\beta$ in $X$ $\kappa$-fellow travel each other with multiplicative constant $n$, then at radius $r$, the distances between points of $\Phi \alpha$ and $\Phi \beta$ is  apart by
\[L(n \cdot \kappa(r)) + \theta(r) \leq  (L n + 1) \kappa(r)
\]
thus  $\Phi \alpha$ and $\Phi \beta$ $\kappa$-fellow travel each other in $Y$. By Proposition~\ref{prop:inverses-SBE}, there exists an inverse, $\overline \Phi$, that is also an $(L, \theta)$-SBE. Thus we have that if $\Phi \alpha$ and $\Phi \beta$ are $n' \cdot \kappa(r)$-tracking, then $\overline \Phi\Phi \alpha$ and $\overline \Phi\Phi \beta$ are $(L n' + 1) \kappa(r)$-tracking each other by the preceding argument. Furthermore, $\alpha, \beta$ are $\theta(r)$ tracking  $\overline \Phi\Phi \alpha$ and $\overline \Phi\Phi \beta$, respectively. Thus $\alpha, \beta$ are tracking each other with distance at most
\[
 (L n' + 1) \kappa(r) + \theta(r) + \theta(r) \leq  (L n' + 3) \kappa(r). \qedhere
\]
\end{proof}

\begin{definition}\label{defn-induced}
It follows from Proposition~\ref{lem1}
 that two quasi-geodesics $\zeta$ and $\xi$ in $X$ $\kappa$--fellow 
travel each other if and only if $\Phi \zeta$ and $\Phi \xi$ $\kappa$-fellow travel each other in $Y$. 
Also by  Corollary~\ref{imageMorse}, the property of being in Morse-neighbourhood of a Morse geodesic ray is preserved under an $(L, \theta)$-SBE. Hence, 
$[\zeta] \in  \partial_\kappa X$ if and only if $[\Phi \zeta] \in \partial_\kappa Y$.  We write $\bfa = [\zeta]$ and $\bfb = [\Phi \zeta]$ and we thus define 
\[
\Phi_\star (\bfa) =\bfb.
\]
\end{definition}


\begin{theorem}\label{invarianttopology}
Consider proper geodesic metric spaces $X$ and $Y$, let $\Phi \from X \to Y$ be a 
$(L, \theta)$-sublinear bi-Lipschitz equivalence between $X$ and $Y$ and let $\overline \Phi$ denote its inverse as in Proposition~\ref{prop:inverses-SBE}. Let $\Phi_\star$ be defined by Definition~\ref{defn-induced}. Then 
for every sublinear function $\kappa$ that dominates $\theta$, $\Phi_\star \from \partial_\kappa X \to \partial_\kappa Y$ is a homeomorphism. 
\end{theorem}

\begin{proof}

By Definition~\ref{defn-induced}, $\Phi_\star$ defined as above gives a bijection between $\partial_\kappa X$ 
and $\partial_\kappa Y$. We need to show that $(\Phi_\star)^{-1}$ is continuous. Then, the 
same argument applied in the other direction will show that $\Phi_\star$ is also continuous 
which means $\Phi_\star$ is a homeomorphism. 

Let $\calV$ be an open set in $\pka X$, $\bfb_X \in \calV$ and 
$\calU_{\kappa}(\bfb_X, \rr)$ be a neighbourhood of $\bfb_X$ that is contained in $\calV$. 
Let $\bfb_Y = \Phi_\star(\bfb_X)$. We need to show that there is a constant $\rr'$ 
such that, for every point $\bfa_Y \in \calU_\kappa(\bfb_Y, \rr')$, we have 
\[
(\Phi_\star)^{-1}(\bfa_Y) = \bfa_X \in \calU_\kappa(\bfb_X, \rr).
\]


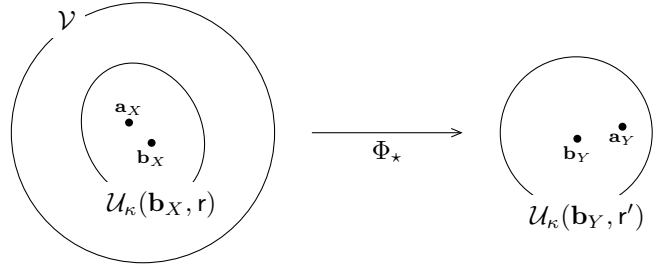
\begin{figure}[h]
\begin{tikzpicture}[line cap=round,line join=round,>=angle 45,x=0.5cm,y=0.5cm]
\clip(-3,-4) rectangle (17,4);
\draw [rotate around={-63.43:(2.5,0)}] (2.5,0) ellipse (0.96cm and 0.78cm);
\draw [rotate around={0:(2.5,0)}] (2.5,0) ellipse (1.75cm and 1.73cm);
\draw(14.02,0) circle (1.01cm);
\draw (-0.02,3.5) node[stuff_fill,anchor=north west] {$\mathcal V $};
\draw (1.31,-1.31) node[stuff_fill,anchor=north west] {$ \mathcal U_{\kappa}(\mathbf b_X, \mathsf r)  $};
\draw (12.55,-1.65) node[stuff_fill,anchor=north west] {$ \mathcal U_{\kappa}(\mathbf b_Y, \mathsf r') $};
\begin{scriptsize}
\fill [color=black] (2.73,-0.27) circle (1.5pt) node[anchor=north] {$\mathbf b_X$} ;
\fill [color=black] (2.13,0.27) circle (1.5pt)  node[anchor=south] {$\mathbf a_X$} ;
\fill [color=black] (14.05,-0.16) circle (1.5pt)  node[anchor=north] {$\mathbf b_Y$} ;
\fill [color=black] (15.25,0.16) circle (1.5pt)  node[anchor=north] {$\mathbf a_Y$} ;
\end{scriptsize}
\draw [color = black,->] (7,0) -- (11,0);
\draw [color = black] (9,-0.5) node[anchor=center] {$\Phi_\star$};
\end{tikzpicture}
\caption{$\Phi_\star$ is a homeomorphism.}
\end{figure}

By Proposition~\ref{Morseray}, if $\zeta$ is a $(\qq, \sQ)$--quasi-geodesic ray in $X$ then $\Phi \zeta$ is in a $\kappa$-neighbourhood of a unique $\kappa$-Morse geodesic ray, up to sublinear fellow travelling. Fix a particular geodesic representative $c$ from that class $[\Phi \zeta]$ and let $(\qq', \sQ')$ be the smallest pair of constants such that 

\[ \Phi \zeta \in \calN_\kappa(c, \mm_{b_Y}(\qq', \sQ')). \]

By Proposition~\ref{imageMorse} $\mm_{b_Y}(\qq', \sQ')$ only depends on $q, Q, L, \theta$ and $\zeta$. Thus we can consider a function $\Phi' \from \RR^2 \to \RR^2$ that takes (\qq, \sQ) to $(\qq', \sQ')$ by the map $\Phi$ on $\kappa$-Morse quasi-geodesic rays. Roughly speaking, $\Phi \zeta$ "can be thought of" as a $(\qq', \sQ')$-quasi-geodesic ray in terms of their sublinear deviation from the geodesic representative. 


Next, let $b_Y$ be the representative geodesic ray in $\bfb_Y$ and let $\mm_{b_X}$ and 
$\mm_{b_Y}$ be their $\kappa$-Morse gauges respectively. 
By Lemma~\ref{L:symmetry}, there is
a constant $\nn_1$ depending on $L, \kappa$ such that 
\[
\Phi b_X \subset \calN_\kappa(b_Y, \nn_1). 
\]
Next, we let 
\[
\nn = L \big( \mm_{b_Y}(\qq', \sQ') + \nn_1\big) (L + 1) + 1
\]
and let $\sR = \sR(b_X, \rr, \nn, \kappa)$ as in Definition~\ref{defn:Morse}. Choose $\rr'$ such that $\rr' \geq L \,\sR + \kappa(r)$. Now let  $\alpha \in \bfa_X$ be a $(\qq, \sQ)$--quasi-geodesic in $X$, where
\begin{enumerate}
\item $\Phi \alpha$ is a ray that is associated with an element in $\pka Y$ that is in $\calU_{\kappa}(\bfb_Y, \rr')$, and 
\item $\qq, \sQ$ is such that  $\mm_{b_X}(\qq, \sQ)$ is small compared to $\rr$.
\end{enumerate}
Pick $x \in \alpha_X|_{\sR}$. Then $\Phi x \in \Phi \alpha|_{\rr'}$ and we have 
\begin{align*} 
d_X(x, b_X) &\leq L (d_Y(\Phi(x), \Phi b_X) + \theta(x)\\
&\leq L \Big(d_Y(\Phi(x), b_Y) + d(\Phi(x)_{b_Y}, b_X) \Big) + \theta(x)\\
& \leq L \Big(d_Y(\Phi(x), b_Y) + \nn_1 \cdot \kappa(\Phi x) \Big) + \theta(x)\\
& \leq L \big( \mm_{b_Y}(\qq', \sQ')+ \nn_1\big) \cdot \kappa(\Phi x) + \theta(x)\\
& \leq L \big( \mm_{b_Y}(\qq', \sQ') + \nn_1\big) \cdot \kappa(\Phi x) + \kappa(x)
\end{align*} 
We also have 
\[
\kappa(\Phi x) \leq L \kappa(x) + \theta(x) \leq (L + 1) \kappa(x)
\]
Combine the preceding inequalities we have
\begin{align*} 
d_X(x, b_X) &\leq L \big( \mm_{b_Y}(\qq', \sQ') + \nn_1\big) \cdot  (L + 1) \kappa(x) + \kappa(x)\\
                    &\leq \Big (L \big( \mm_{b_Y}(\qq', \sQ') + \nn_1\big) (L + 1) +1 \Big) \kappa(x)
\end{align*}
imply that 
\[
\alpha|_{\sR} \subset \calN_\kappa(b_X, \nn). 
\]
Now, Definition~\ref{defn:Morse} implies that 
\[
\alpha|_\rr \subset \calN_\kappa(b_X, \mm_{b_X}). 
\]
Therefore, $\bfa_X \in \calU_{\kappa}(\bfb_X, \rr)$ and
\[
(\Phi_\star)^{-1} \calU_{\kappa}(\bfb_Y, \rr') \subset \calU_{\kappa}(\bfb_X, \rr). 
\]
But $ \calU_{\kappa}(\bfb_Y, \rr') $ contains an open neighbourhood of $\bfb_Y$, therefore, 
$\bfb_Y$ is in the interior of $\Phi \calV$. This finishes the proof. 
\end{proof}

\subsection{An application}
\label{subsec:application-of-thA}

In \cite{Be19}, Behrstock considers an infinite family of right-angled Coxeter groups. 
The definining graphs for one of the smallest pair of Behrstock's groups are as follows. $\Gamma_{14}$ is the simple graph on $28$ vertices $\{a_1,a_2,\ldots ,\,a_{14}, b_{14}\}$ where for $c,c' \in \{a,b \}$ there is an edge between $c_i$ and $c'_j$ if and only if $\vert i - j \vert =1$.
$\Gamma$ is the graph $\Gamma_{14}$ to which we add a $5$-cycle $a_1 -a_4 - a_7 - a_{10} - a_{13}$.

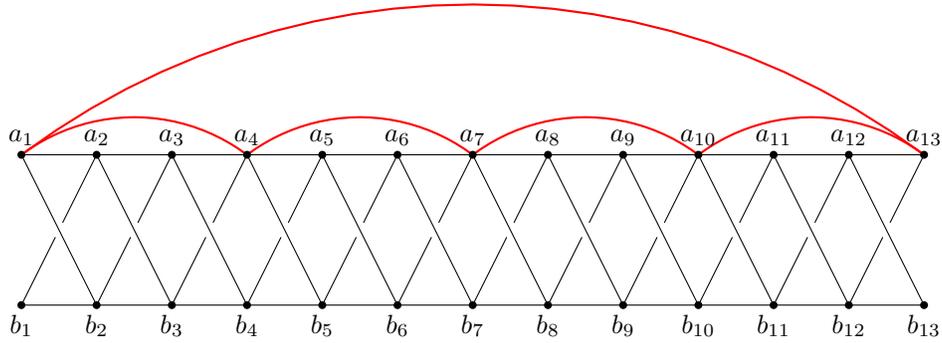
\begin{figure}[h!]

\begin{center}
\begin{tikzpicture}[line cap=round,line join=round,>=triangle 45,x=1cm,y=1cm]
\clip(-8,-3) rectangle (8,4);

\draw [shift={(-4.5,-1)}, thick, red] plot[domain=0.93:2.21,variable=\t]({1*2.5*cos(\t r)+0*2.5*sin(\t r)},{0*2.5*cos(\t r)+1*2.5*sin(\t r)});
\draw [->] (-2,-4) -- (1,-4);
\draw [shift={(-1.5,-1)}, thick, red] plot[domain=0.93:2.21,variable=\t]({2.5*cos(\t r)},{1*2.5*sin(\t r)});
\draw [shift={(1.5,-1)}, thick, red] plot[domain=0.93:2.21,variable=\t]({1*2.5*cos(\t r)},{1*2.5*sin(\t r)});
\draw [shift={(4.5,-1)}, thick, red] plot[domain=0.93:2.21,variable=\t]({1*2.5*cos(\t r)},{1*2.5*sin(\t r)});
\draw [shift={(0,-7)}, thick, red] plot[domain=0.93:2.21,variable=\t]({1*10*cos(\t r)},{10*sin(\t r)});
\foreach \i in {-6,-5,-4,-3,-2,-1,0,1,2,3,4,5}
{\draw (\i,-1)-- (\i+1,-1);
\draw (\i,1)-- (\i+1,1);
\draw (\i,-1)-- (\i+0.45,-0.1);
\draw (\i+1,1)-- (\i+0.55,0.1);
\draw (\i,1)-- (\i+1,-1);
\pgfmathtruncatemacro{\j}{\i+7};
\fill [color=black] (\i,1) circle (1.5pt) node[above]{$a_{\j}$};
\fill [color=black] (\i,-1) circle (1.5pt) node[below]{$b_{\j}$};
}
\fill [color=black] (6,1) circle (1.5pt) node[above]{$a_{13}$};
\fill [color=black] (6,-1) circle (1.5pt) node[below]{$b_{13}$};
\end{tikzpicture}
\caption{The graph $W$ (when compared to $W_{13}$, additional edges are colored).}
\end{center}
\end{figure}

Let $W_{13}$ and $W$ be the right-angled Coxeter groups with defining graphs $\Gamma_{13}$ and $\Gamma'$ as above respectively.

\begin{lemma}
$W$ and $W_{13}$ have quadratic divergence. In particular, they are not hyperbolic relative to any of their subgroups of infinite index, and their asymptotic cones are not tree-graded spaces.
\end{lemma}

The fact that $W$ has quadratic divergence and is not relatively hyperbolic was checked by Behrstock \cite{Be19}. We follow the same method, and apply it to $W_{13}$ as well below.

\begin{proof}
Let us prove that $W$ and $W_{13}$ have the Dani-Thomas $\mathcal {CFS}$ property, defined in \cite[4.2]{DT}.
Dani and Thomas associate new graphs to the defining graphs of right-angled Coxeter groups. 
Let us denote $\Lambda$ and $\Lambda_{13}$ the graphs associated to $\Gamma$ and $\Gamma_{13}$ respectively. 
The vertex sets in $\Lambda$ and in $\Lambda_{13}$ are the four-cycles in $\Gamma$ and $\Gamma_{13}$ respectively, and there is an edge between $\lambda$ and $\lambda'$ if $\lambda$ and $\lambda'$ share a common edge.
In the graph $\Lambda$, there are $56$ vertices, described in terms of loops in $\lambda$ as follows:
\begin{eqnarray*}
\lambda_i^a  = a_i - a_{i+1} - a_{i+2} - b_{i+1} & i=1,\ldots ,11 \\
\lambda_i^b  = b_i - b_{i+1} - b_{i+2} - a_{i+1} & i=1,\ldots ,11 \\
{\lambda'}_i^a  = a_i - a_{i+1} - b_{i+2} - b_{i+1} & i=1,\ldots ,11 \\
{\lambda'}_i^b  = b_i - b_{i+1} - a_{i+2} - a_{i+1} & i=1,\ldots ,11. \\
\mu_i = a_i - a_{i+1} - b_i - b_{i+1} & i = 1,\ldots, 12
\end{eqnarray*}
There is an edge $\lambda_i^a$ and ${\lambda'}_j^a$ and between $\lambda_i^b$ and ${\lambda'}_j^b$ if and only if  $0 \leqslant j -i \leqslant 1$, an edge between ${\lambda'}_i^a$ and ${\lambda'}_j^{b}$ if and only if $\vert j -i \vert = 1$, an edge between $\lambda_i^a$ and ${\lambda'}_j^b$ and between $\lambda_i^b$ and ${\lambda'}_j^a$ if and only if $0 \leqslant i - j \leqslant 1$, and, finally, an edge between $\mu_i$ and $\lambda^{a}_j$ or $\lambda^b_j$ if $0\leqslant i - j \leqslant 1$. It is clear from this that $\Lambda$ is connected. Moreover, any vertex of $\Gamma$ is present in at least one loop that is represented by a vertex in $\Lambda$, that is, using the words of Dani-Thomas, the single component of $\Lambda$ has full support. This is precisely the property $\mathcal{CFS}$.

It now follows from \cite{DT} that the groups $W$ and $W_{13}$ have quadratic divergence.
Since the relatively hyperbolic groups have exponential divergence $W$ and $W_{13}$ are not relatively hyperbolic \cite{Sistorel}, and their asymptotic cones are not tree-graded spaces \cite{DS}.

Let's now consider the graph $\Lambda_{13}$. As compared to $\Lambda$, there are sixteen additional loops, namely
\begin{align*}
a_i -a_{i+3} - a_{i+2} - a_{i+1} \\
a_i - a_{i+3} - b_{i+2} - b_{i+1} \\
a_i - a_{i+3} -b_{i+2} - a_{i+1} \\
a_i - a_{i+3} - a_{i+2} - b_{i+1}
\end{align*}
for $i \in \{ 1, 4, 7, 10 \}$.
All these new loops are easily seen to be connected in $\Lambda_{13}$ to one of the $\lambda_i^a$, $\lambda_i^b$, ${\lambda'}_i^a$, ${\lambda'}_i^b$. It follows that $W_{13}$ has the $\mathcal{CFS}$ property.
\end{proof}
\begin{lemma}\label{totaldisconnect}
Let $\kappa$ be a sublinear function, $\pka W_{13}$ is totally disconnected. 
\end{lemma}
\begin{proof}
Let $c_i = a_ib_i$, then it is known that the following graph defining a right-angled Artin group  on the vertices embeds in  $W_{13}$ as a finite index subgroup \cite{DJ00}.\\
\begin{figure}[h!]
\begin{tikzpicture}[scale=0.7]
 \tikzstyle{vertex} =[circle,draw,fill=black,thick, inner sep=0pt,minimum size=1.5pt]
[thick, 
    scale=1,
    vertex/.style={circle,draw,fill=black,thick,
                   inner sep=0pt,minimum size= 1.5 pt},
                  
      trans/.style={thick,->, shorten >=6pt,shorten <=6pt,>=stealth},
   ]

  \node[vertex] (c1) at (1,0) [label=above:$c_1$] {}; 
    \node[vertex] (c2) at (2,0) [label=above:$c_2$] {}; 
      \node[vertex] (c3) at (3,0) [label=above:$c_3$] {}; 
        \node[vertex] (c4) at (4,0) [label=above:$c_4$] {}; 
          \node[vertex] (c5) at (5,0) [label=above:$c_5$] {}; 
    \node[vertex] (c6) at (6,0) [label=above:$c_6$] {}; 
      \node[vertex] (c7) at ( 7,0) [label=above:$c_7$] {}; 
        \node[vertex] (c8) at (8,0) [label=above:$c_8$] {}; 
          \node[vertex] (c9) at (9,0) [label=above:$c_9$] {}; 
    \node[vertex] (c10) at (10,0) [label=above:$c_{10}$] {}; 
      \node[vertex] (c11) at (11,0) [label=above:$c_{11}$] {}; 
        \node[vertex] (c12) at (12,0) [label=above:$c_{12}$] {}; 
           \node[vertex] (c13) at (13,0) [label=above:$c_{13}$] {}; 
  
  \draw[thick] (c1)--(c2)--(c3)--(c4)--(c5)--(c6)--(c7)--(c8)--(c9)--(c10)--(c11)--(c12)--(c13){};
  \end{tikzpicture}
  \caption{Right-angled Artin group defined by a path, $A_{13}$.}
  \end{figure}
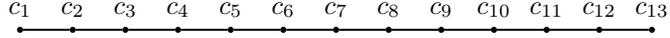

By \cite[Theorem 1.1]{IZ}, $\pka A_{13}$ continuously injects into the Gromov boundary of the contact graph of the associated Salvetti complex. Since contact graphs are quasi-isometric to trees \cite{Hagen14}, their Gromov boundaries are totally disconnected. Thus $\pka A_{13}$ is totally disconnected. Since $W_{13}$ is quasi-isometric to $A_{13}$, we have that $\pka W_{13}$ is also totally disconnected.
\end{proof}

\begin{proposition}
$W$ and $W_{13}$ are not sublinear biLipschitz equivalent.
\end{proposition}

\begin{proof}
The $5$-cycle in $\Gamma$ induces an embedding of a finite-index subgroup of the Coxeter group $\Delta = \langle r,s, t \mid (rs)^5, (st)^2, (tr)^4 \rangle$ in $W$. $\Delta$ itself is a finite-index subgroup of a hyperbolic surface group. Behrstock proves that the embedding in $W$ is stable. Since $\Delta$ is a finite index subgroup of a hyperbolic group, the sublinearly Morse boundary of $\Delta$ is also a circle. On the other hand, by Lemma~\ref{totaldisconnect}, $\pka W_{13}$ is totally disconnected and thus it does not contain a circle. Therefore $\pka W$ and $\pka W_{13}$ are not homeomorphic, and thus we conclude from Theorem~\ref{invarianttopology} that $W$ and $W_{13}$ are not sublinear biLipschitz equivalent.
\end{proof}

It follows from the work of Hagen, Kazachkov and Casals-Ruiz \cite{HKR} that $W$ and $W_{13}$ have unique asymptotic cones.
Up to our knowledge, it was not known whether the asymptotic cones of $W_{13}$ and $W$ are biLipschitz homeomorphic or not. 

\begin{remark}
It is also worth noting that even though we distinguish the $\kappa$-Morse boundaries of $W_{13}$ and $W$ by the Morse boundaries that are proper subspaces of their $\kappa$-Morse boundaries, Morse boundaries in general are not SBE-invariant, therefore it was necessary to invoke Theorem~\ref{invarianttopology}.
\end{remark}

\section{Random walk on groups and sublinear rays}
\label{sec:random-walks}

\subsection*{Random walks}
Let $G$ be a countable group, and let $\mu$ be a probability measure on a symmetric generating set of $G$.
We consider the \emph{step space} $(G^\mathbb{N}, \mu^\mathbb{N})$, whose elements we denote as $(g_n)$. 
The \emph{random walk driven by }$\mu$ is the $G$-valued stochastic process $(w_n)$, where for each $n$ we define the product
$$w_n := g_1 g_2 \dots g_n.$$
We denote as $(\Omega, \mathbb{P})$ the \emph{path space}, i.e. the space of sequences $(w_n)$, where $\mathbb{P}$ is the measure induced by pushing forward the measure $\mu^\mathbb{N}$ from the step space. Elements of $\Omega$ are called \emph{sample paths} and will be also denoted as $\omega$. Finally, let $T : \Omega \to \Omega$ be the left shift on the path space. 

\subsection*{Background on boundaries}
Let us recall some fundamental definitions from the boundary theory of random walks. For a more extensive exposition, see \cite{Kai00}. 
Let $(B, \mathcal{A})$ be a measurable space on which $G$ acts by measurable isomorphisms;  a measure $\nu$ on $B$ is $\mu$-\emph{stationary} if $\nu = \int_G g_\star \nu \ d\mu(g)$, and in that case the pair $(B, \nu)$ is called a $(G, \mu)$-\emph{space}. 
Recall that a \emph{$\mu$-boundary} is a measurable $(G, \mu)$-space $(B,\nu)$ such that there exists 
a $T$-invariant, measurable map $\textbf{bnd} : (\Omega, \mathbb{P}) \to (B, \nu)$, called the \emph{boundary map}. 

Moreover, a function $f: G \to \mathbb{R}$ is $\mu$-\emph{harmonic} if $f(g) = \int_G f(gh) \ d\mu(h)$ for any $g \in G$. 
We denote by $H^\infty(G, \mu)$ the space of bounded, $\mu$-harmonic functions. 
One says a $\mu$-boundary is the \emph{Poisson boundary} of $(G, \mu)$ if the map 
\[
\Psi : H^\infty(G, \mu) \to L^\infty(B, \nu)
\]
given by $\Psi(f)(g) := \int_B f \ dg_\star \nu$ is a bijection. 
The Poisson boundary $(B, \nu)$ is the maximal $\mu$-boundary, in the sense that for any other $\mu$-boundary $(B', \nu')$ there exists a $G$-equivariant, 
measurable map $p: (B, \nu) \to (B', \nu')$. The result of this section concerns the shape of all sample paths.

\begin{theorem}\label{randomwalkisthetaray}
Let $G$ be the mapping class group $Mod(S)$ of a finite type surface, or let $G$ be a relatively hylic group. Let $\mu$ be a probability measure on G with finite first moment with respect to the metric $d$, such that the semigroup generated by the support of $\mu$ is a non-amenable group. Let $\kappa(r) = \log (2+r)$. Then there exists a constant $A$ such that almost every sample path $(w_n)$ is such that $(w_n \go)$ is a $\kappa$-sublinear ray. Moreover, the upper and lower large-scale Lipschitz constants of $(w_n)$ are equal, that is, there exists $A$ such that 
\begin{equation}
    A \vert n - m \vert - \kappa(\max (n,m)) \leqslant d(w_n, w_m) \leqslant A \vert n-m \vert + \kappa(\max (n,m)).
    \label{eq:randomwalk}
\end{equation}

\end{theorem}

\begin{proof}
By Theorem $C$ in \cite{QRT20}, 
\[
\limsup_{n \to \infty} \frac{d_w(w_n, \gamma_\omega)}{\log n} < +\infty.
\]
Thus there exists a $C$ such that 
\[
\limsup_{n \to \infty} d_w(w_n, \gamma_\omega) \leq C \log n
\]
By \cite{MaherTiozzo}, weakly hyperbolic groups have positive drift on their respective associated hyperbolic  spaces. And by the Distance Formula (Masur-Minsky \cite{MM00}),  distance (to the origin) in the random walk on the group is coarsely bounded by the distance to the origin in the associated curve graph $d_S$. Thus random walks on a mapping class groups have positive drifts (\cite{MaherTiozzo}) That is to say, there exists an $A$ such that for $n$ large enough
\[
d(w_n \go, \go) \geq An.
\]
Let $p_n$ denote the nearest point projection of $w_n$ to $\gamma_\omega$. We have by triangle inequality
\[ d(\go, p_n) \geq d(w_n \go, \go) -d(w_n \go, p_n)) \geq An -C \log n.\]
We have that as $n \to \infty$
\[
An \geq \lim d(\go, p_n) \lim \geq An -C \log n = An = d(\go, \gamma(An)),
\]
where the last equality comes from the fact that $\gamma$ is unit speed. Therefore 
\[ d(w_n\go, p_n) = C \log n = d(w_n \go, \gamma(An)).\]
Next we construct a map on $\gamma$ such that for each $i \in \NN$ and $t\in [i-\frac 12, i+\frac 12)$  we define:
\[\mathcal T(\gamma(At)) := w_i.\]
We need to show that this is a sublinear bi-Lipschitz equivalence. For any given $t, t'$, assume that $t\in [i-\frac 12, i+\frac 12)$ and $t' \in [j-\frac 12, j+\frac 12)$ and also assume without loss of generality that $j-i$, 
we have that 
\begin{align*}
d(\mathcal T(\gamma(At), \gamma(At')) &= d(w_i, w_j) \leq d(w_i, \gamma(Ai)) + |j-i|A + d(w_j, \gamma(Aj))\\
                                                     &\leq A|j-i| + 2C \log \Vert w_j \Vert \\
                                                     &\leq A|t-t'+1| + 2C \log \Vert w_j \Vert \\
                                                     &\leq A|t-t'| + 2C \Vert \log w_j \Vert +1
\end{align*}
On the other hand we have 
\begin{align*}
d(\mathcal T(\gamma(At), \gamma(At')) &= d(w_i, w_j) \\
                                                     &\geq A|j-i| - d(w_i, \gamma(Ai)) - d(w_i, \gamma(Ai)) \\
                                                     &\geq A |j-i|-2C \log j \\
                                                     &\geq  A |t-t'|-A -2C \log \Vert w_j \Vert.
                                                     \end{align*}
Thus $\mathcal T$ is an SBE and the image of the geodesic $\gamma_\omega$ contains the sample path and is an $(A, \log n )$-ray. The proof for $G$ is a relatively hylic group is identical, using also the facts that relatively hylic group is weakly hyperbolic and there is also a distance formula that is similar to that of the mapping class group \cite{sistopaper}.
\end{proof}

More generally, aside from the aforementioned two groups, the conclusion can be applied to a wider range of countable groups with a compact  boundary. We say a boundary is \emph{stably visible} if any sequence of geodesics whose endpoints converges to two distinct points on the boundary intersects some bounded set of $X$.
We can then combine the same argument as Theorem~\ref{randomwalkisthetaray} and apply Theorem 6 in \cite{Ti15} to obtain the following:
\begin{theorem}\label{countablegroups}
 Let $G$ be a countable group acting via isometries on a proper, geodesic, metric space $(X,d)$ with a non-trivial, stably visible compactification. Let $\mu$ be a probability measure on G with finite first moment with respect to $d$, such that the semigroup generated by the support of $\mu$ is a non-amenable group. Then there exists a constant $A$ and a sublinear function $\kappa$ such that almost every sample path $(w_n)$ is such that $(w_n \go)$ is a $(A, \kappa)$-biLipschitz ray.
\end{theorem}
Example of such compactifications includes but are not limited to:
\begin{enumerate}
\item the hyperbolic compactification of Gromov hyperbolic spaces;
\item the end compactification of Freudenthal and Hopf \cite{Hop44};
\item the Floyd compactification (Section 3.2, \cite{Ti15});
\item the visual compactification of a large class of CAT(0) spaces (Section 3.4, \cite{Ti15}).
\item the redirecting compactification of asymptotically tree-graded spaces (\cite{QR23}).
\end{enumerate}


\begin{thebibliography}{}

\end{thebibliography}


\begin{thebibliography}{QRT200}

\bibitem[Be19]{Be19}
J. Behrstock,
\newblock {A counterexample to questions about boundaries, stability, and
              commensurability},
\newblock in \emph{Beyond hyperbolicity},
\newblock
{London Math. Soc. Lecture Note Ser.},
{454},
{151--159},
\newblock {Cambridge Univ. Press, Cambridge, 2019}.


\bibitem[Be]{BeP}
J. Behrstock,
\newblock {Personnal communication}.
   

\bibitem[BH99]{BH99}
M. Bridson and A. Haefliger,  
\newblock{\em Metric spaces of non-positive curvature},
\newblock{Grundlehren der Mathematischen Wissenschaften [Fundamental Principles of Mathematical Sciences], 319.}
Springer-Verlag, Berlin, 1999. xxii+643. ISBN 3-540-64324-9.

\bibitem[Cho]{Choi}
Inhyeok Choi
\newblock{\em Limit laws on Outer space, Teichmm\"uller space and CAT(0) spaces}
\newblock{arXiv:2207.06597 }
\bibitem[CK00]{CK00}
C. B. Croke and B. Kleiner, 
\newblock{\em Spaces with nonpositive curvature and their ideal boundaries},
\newblock{Topology 39 (2000), no. 3, 549--556.}


\bibitem[Cor08]{Cor08}
Yves Cornulier.
\newblock {\em Dimension of asymptotic cones of Lie groups.}
\newblock { J. Topology}, {\bf 1} (2), 343--361, 2008. 


\bibitem[Cor11]{Cor11}
Yves Cornulier.
\newblock {\em Asymptotic cones of Lie groups and cone equivalences}
\newblock { Illinois J. Math.}, 55 (1), 237--259, 2011. 

\bibitem[Cor19]{cornulier2017sublinear}
Yves Cornulier.
\newblock {\em On sublinear bilipschitz equivalence of groups.}
\newblock {Ann. Sci. \'{E}c. Norm. Sup\'{e}r. (4)}, 52(5):1201--1242, 2019.



\bibitem[DJ00]{DJ00}
Michael W. Davis and Tadeusz Januszkiewicz.
\newblock {\em Right-angled Artin groups are commensurable with right-angled Coxeter groups. }
\newblock {J. Pure Appl. Algebra } 153 (2000), no. 3, 229–235. 


\bibitem[DK18]{DrutuKapovich}
Cornelia Dru\c{t}u and Michael Kapovich.
\newblock {\em Geometric group theory}, volume~63 of {\em American Mathematical
  Society Colloquium Publications}.
\newblock American Mathematical Society, Providence, RI, 2018.
\newblock With an appendix by Bogdan Nica.

\bibitem[Dru02]{Drutu}
Cornelia Dru\c{t}u 
\newblock {\em Quasy-isometry invariants and asymptotic cones}, 
\newblock  Int. J. Alg. Comp. 12 (2002), 99-135.

\bibitem[DS]{DS}
Cornelia Dru\c{t}u and Mark Sapir,
\newblock {\em Tree-graded spaces and asymptotic cones of groups},
\newblock Topology 44 (2005) 959 - 105.

\bibitem[DT15]{DT}
Pallavi Dani and Anne Thomas,
\newblock {\em Divergence in right-angled Coxeter groups.} (English summary)
\newblock Trans. Amer. Math. Soc. 367 (2015), no. 5, 3549-3577.

\bibitem[GQR22]{GQR22}
Ilya Gekhtman, Yulan Qing and Kasra Rafi
\newblock{\em QI-invariant model of Poisson boundaries of $\CAT$ groups}
\newblock{arXiv:2208.04778}


\bibitem[Gro87]{Gromov87HypGrp}
Mikhael~L. Gromov.
\newblock {\em Hyperbolic groups.}
\newblock {Essays in group theory}, volume~8 of {\em Math. Sci. Res.
  Inst. Publ.}, pages 75--263. Springer, New York, 1987.

\bibitem[Gro93]{AsInv}
Mikhael~L. Gromov.
\newblock {\em Asymptotic invariants of infinite groups.}
\newblock { Geometric group theory, {V}ol.\ 2 ({S}ussex, 1991)}, volume
  182 of {\em London Math. Soc. Lecture Note Ser.}, pages 1--295. Cambridge
  Univ. Press, Cambridge, 1993.
  
\bibitem[Hag14]{Hagen14}
Mark Hagen.
\newblock {\em Weak hyperbolicity of cube complexes and quasi-arboreal groups.}
\newblock { J. Topol. } 7:2, 385--418, 2014. 


\bibitem[HKR]{HKR}
Mark Hagen, Ilya Kazachkov and Montserrat Casals-Ruiz   Working draft, 2022.
\newblock {\em Real cubings and asymptotic cones of hierarchically hyperbolic groups.}
\newblock{Working draft, 2022.}
  
 \bibitem[Hop44]{Hop44}
  Hopf, H., 
 \newblock  {\em Enden offener R\"aume und unendliche diskontinuierliche Gruppen,}
 \newblock{ Comment. Math. Helv.} 16 (1944), 81–100.
 
 
 \bibitem[IZ]{IZ}
Merlin Incerti-Medici and Abdul Zalloum
 \newblock  {\em Sublinearly Morse boundaries from the view point of combinatorics}
  \newblock  { arXiv:2101.01037}
 

\bibitem[Kai00]{Kai00}
V. Kaimanovich, 
\newblock{\em The Poisson formula for groups with hyperbolic properties},
\newblock{Ann. of Math. (2) 152 (2000), no. 3, 659--692.}

%

\bibitem[MT18]{MaherTiozzo}
J. Maher and G. Tiozzo, 
\newblock{Random walks on weakly hyperbolic groups}, 
\newblock{J. Reine Angew. Math.} 742, (2018), 187--239.

\bibitem[MM00]{MM00}
H. Masur  and Y. Minsky,  
\newblock{\em Geometry of the complex of curves. II. Hierarchical structure},
\newblock{Geom. Funct. Anal. }10 (2000), no. 4, 902--974.

\bibitem[MQZ21]{MQZ21}
Devin Murray, Yulan Qing and Abdul Zalloum
\newblock{\em Sublinearly Morse geodesics in CAT(0) spaces: lower divergence and hyperplane characterization}
\newblock {to appear in Algebraic \& Geometric Topology.}


\bibitem[Mor24]{Mor24}
H. M. Morse, 
\newblock{\em A fundamental class of geodesics on any closed surface of genus greater than one}, 
\newblock{Trans. Amer. Math. Soc.} 26 (1924), 25--60.



\bibitem[NQ22]{NQ22}
Hoang Thanh Nguyen, Yulan Qing
\newblock{\em Sublinearly Morse Boundary of CAT(0) admissible groups}
\newblock{arXiv:2203.00935}


\bibitem[Pal20]{pallierHYPSBE}
Gabriel Pallier.
\newblock Large-scale sublinearly {L}ipschitz geometry of hyperbolic spaces.
\newblock {\em J. Inst. Math. Jussieu}, 19(6):1831--1876, 2020.

\bibitem[Pau96]{PaulinGrpHypBord}
Fr\'{e}d\'{e}ric Paulin.
\newblock Un groupe hyperbolique est d\'{e}termin\'{e} par son bord.
\newblock {\em J. London Math. Soc. (2)}, 54(1):50--74, 1996.



\bibitem[Qin16]{qing1}
Y. Qing,
\newblock{\em Geometry of Right-Angled Coxeter Groups on the Croke-Kleiner Spaces},
\newblock{Geom. Dedicata }183 (2016), no. 1, 113--122.

\bibitem[QRT19]{QRT19}
Y. Qing, K. Rafi and G. Tiozzo,
\newblock{\em Sublinearly Morse boundary I: CAT(0) spaces},
\newblock {Advances in Mathematics}, 404, 108442. 

\bibitem[QRT20]{QRT20}
Yulan Qing, Kasra Rafi, and Giulio Tiozzo.
\newblock {\em Sublinearly Morse Boundary II: Proper geodesic spaces}
\newblock {arXiv:2011.03481}


\bibitem[QR23]{QR23}
Yulan Qing and Kasra Rafi.
\newblock {\em Compactification of Asymptotically Tree-graded Spaces},
\newblock {in preparation} 

\bibitem[QZ19]{QZ23}
Y. Qing and A. Zalloum, 
\newblock{\em Geometry and dynamics on sublinearly Morse boundaries of CAT(0) groups}, 
\newblock{arxiv.org/abs/1911.03296 }



%

\bibitem[Sis12]{Sistorel}
A. Sisto,
\newblock {\em On metric relative hyperbolicity},
\newblock arXiv:1210.8081.

\bibitem[Sis13]{sistopaper}
A. Sisto, 
 \newblock{\em Projections and relative hyperbolicity},
 \newblock{Enseign. Math. }(2) 59 (2013), no. 1-2, 165--181.

%


\bibitem[Tio15]{Ti15}
G. Tiozzo, 
\newblock{\em Sublinear deviation between geodesics and sample paths},
\newblock{Duke Math. J. } 164 (2015), no. 3, 511--539.

\bibitem[Zal21]{Zal21}
Abdul Zalloum
\newblock{\em Convergence of sublinearly contracting horospheres}
\newblock {to appear in Geometric Dedicata.}
\end{thebibliography}
\end{document}